\documentclass[11pt,reqno]{amsart}

\usepackage{amsfonts} 
\usepackage{amsmath,amscd}
\usepackage{amssymb}
\usepackage{tikz-cd}
\usepackage{amsthm}
\usepackage{cases}

\usepackage{galois}
\usepackage{chemarrow}
\usepackage[all]{xy}
\usepackage{graphicx}
\usepackage[colorlinks,
            linkcolor=red,
            anchorcolor=black,
            citecolor=blue
            ]{hyperref}	
\usepackage{mathrsfs}
\usepackage{extarrows}
\usepackage{texnames}
\usepackage{textcase}

\makeatletter
\def\section{%
  \@startsection{section}{1}
    {\z@}
    {2.0ex plus 0.8ex minus .1ex}
    {1.0ex plus .2ex}
    {\large\bfseries\boldmath\centering\MakeTextUppercase}%
}
\makeatother

\def\curlA{\mathcal A}

\def\curlL{\mathcal L}
\def\curlyL{\mathscr L}

\def\grad{\nabla}

\newcommand{\si}{\sigma}
\newcommand{\downto}{\searrow}
\newcommand{\ga}{\gamma}

\newcommand{\ep}{\epsilon}

\newcommand{\vare}{\varepsilon}
\newcommand{\al}{\alpha}
\newcommand{\La}{\Lambda}
\newcommand{\Vol}{{\rm Vol}}
\newcommand{\vol}{{\rm Vol}}
\newcommand{\dist}{{\rm dist}}
\newcommand{\lap}{\Delta}
\newcommand{\de}{\delta}

\newcommand{\boundary}{\partial}
\newcommand{\ti}{\tilde}
\def\of{\circ}
\newcommand{\partt}{\frac{\partial}{\partial t}}
\newcommand{\upto}{\nearrow}

\def\Sc{{\rm  R}}
\def\Rm{{\rm  Rm}}
\def\si{\sigma}
\def\Riem{{\rm  Rm}}
\def\Ric{{\rm Ric}}
\def\Rc{{\rm Ric}}
\def\Ricci{{\rm Ric}}

\def\boundary{ \partial }
\newcommand{\N}{ \mathbb N}
\newcommand{\R}{ \mathbb R}
\newcommand{\B}{ \mathbb B}

\newtheorem{thm}{Theorem}[section]
\newtheorem{lem}[thm]{Lemma}
\newtheorem{emptythm}[thm]{}

\newtheorem{rem}[thm]{Remark}

\newtheorem{defi}[thm]{Definition}

\numberwithin{equation}{section}

\begin{document}
\title[ Solutions to Ricci flow whose scalar curvature is in $L^{p}$]{ Volume estimates  and convergence results  for  solutions to  Ricci flow with  $L^{p}$ bounded scalar curvature}
\author[\sc\small J\MakeLowercase{iawei} L\MakeLowercase{iu and} M\MakeLowercase{iles} S\MakeLowercase{imon}]{\sc\large J\MakeLowercase{iawei} L\MakeLowercase{iu and} M\MakeLowercase{iles} S\MakeLowercase{imon}}
\address{Jiawei Liu\\ School of Mathematics and Statistics\\ Nanjing University of Science and  Technology\\ Xiaolingwei Street 200\\ Nanjing 210094, China} \email{jiawei.liu@njust.edu.cn}
\address{Miles Simon\\ Institut f\"ur Analysis und Numerik\\ Otto-von-Guericke-Universit\"at Magdeburg\\ universit\"atsplatz 2\\ Magdeburg 39106\\ Germany} \email{miles.simon@ovgu.de}
\thanks{The authors are supported by the Special Priority Program SPP 2026 ``Geometry at Infinity" from the German Research Foundation (DFG)}

\maketitle
 
\vskip -3.99ex

\centerline{\noindent\mbox{\rule{3.99cm}{0.5pt}}}

\vskip 5.01ex



\ \ \ \ {\bf Abstract: }
 In this paper we study $n$-dimensional Ricci flows $(M^n,g(t))_{t\in [0,T)},$ where $T< \infty$ is a potentially singular time, and for which the spatial $L^p$ norm,  $p>\frac n 2$,  of the scalar curvature is uniformly  bounded on $[0,T).$ 
 In the case that $M$ is closed and four dimensional,  we explain why    non-collapsing estimates hold and how they can be combined  with    integral bounds on the Ricci and full curvature tensor of  the  prequel paper \cite{LiuSim1}, as well as non-inflating estimates (already known due to works of Bamler \cite{Bam}), to   obtain  an improved space time integral bound of the Ricci curvature.  
  As an application of these estimates,  we show  that if we further restrict to $n=4$, then  the solution convergences to an orbifold as $t \to T$ and  that the flow can be extended    using the Orbifold Ricci flow to the time interval  $ [0,T+\sigma)$ for some $\sigma>0.$
  We also prove local versions of many of the  results mentioned above.

\medskip


\section{Introduction}\label{sec:intro}
 
This paper is the sequel to the paper \cite{LiuSim1}, where various  {\bf localised weighted   integral curvature estimates} were proved  for  four dimensional real Ricci flows and  Kähler Ricci flows in all dimensions. The     localised weighted curvature   integral estimates 
  generalise and improve  those integral formulae proved in \cite{MSIM1}.
  In this paper we use these estimates  as well as other methods and estimates    to 
examine  the structure of possible finite time singular regions for closed real solutions $(M^n,g(t))_{t\in [0,  T)}$ to the  Ricci flow.   

The Ricci flow, 
$\partt g(t) = -2\Rc(g(t)),$ was first introduced and studied by Hamilton in \cite{Ham}.
We are interested in the behaviour of the flow near potentially singular finite times.

In particular,  we are interested  in  the setting: \\
\noindent ({\bf B}):   $(M^n,g(t))_{t\in [0,T)}, 2\leq T < \infty$  is a smooth 
$n$-dimensional  solution   to Ricci flow   with  $\Sc_{g(t)} \geq -1$ for all $t\in  [0,T)$ and $N$  is a connected, smooth   $n$-dimensional submanifold   with smooth boundary $\boundary N$, $N$  compactly contained in $M$,   such that 
$  \Omega :=  \cup_{s=0}^{T} B_{g(s)}(\boundary N,10)  $ has compact closure  in $M$  and  $\sup\limits_{    x \in  \overline{ \Omega }, t  \in [0,T) } |\grad^k \Rm_{g(t)}|_{g(t)} \leq 1$   for all $k\in\{1,\ldots,4\},$\\     $ \sup\limits_{    x \in  \overline{ \Omega }, t  \in [0,T) } |\grad^k \Rm_{g(t)}|_{g(t)}\leq C_k< \infty ,$ for all $k\geq 5,$  $k\in \N,$
and \\$\sup\limits_{x \in \Omega \cup N, t \in [0,1] }  |\Rm_{g(t)}|_{g(t)}\leq 1,$ 
\vspace{1cm}
 
 \vspace{0.5cm}

 See the introduction  of \cite{LiuSim1} for some  equivalent conditions, denoted by (${\rm \bf  A}$) and (${\rm \bf C}$) in \cite{LiuSim1},  to  (${\rm  \bf B}$).

We are interested in the behaviour of the Ricci flow in the local region $N$ as $t\upto     T< \infty$ with $  T$ being a potentially singular time and where a neighbourhood of   $\boundary N$ is regular :  $(N,g(t))_{t\in [0,  T)}$ is also a solution to Ricci flow, which possibly becomes singular at time $  T$, but it is uniformly 
regular (in time) at and near its spatial boundary.

 Under the further assumption that 
\begin{eqnarray*} 
\int_N |\Sc_{g(t)}|^{\frac{n}{2}+\si} dV_{g(t)} \leq L < \infty 
\end{eqnarray*}
for all $t\in [0,  T),$ for some $\si\in (0,1) ,$  
  we show 
\begin{itemize}
\item[(a)]   non-collapsing  estimates, generalising  results of Perelman  \cite{GP}, Zhang \cite{QSZ,QSZ2}   and Hamanaka \cite{SHA}: see Theorem  \ref{mainthmIntro}. 
We also give a proof of a non-inflating estimate  in this setting (see Section 3), by generalising the proof of Zhang \cite{QSZ1}, who proved non-inflating in a bounded scalar curvature setting.    The non-inflating estimate is  known to hold in a much more  general setting due to the work of Bamler (see Theorem 8.1 in \cite{Bam}).
Also, local   non-collapsing estimates  for points near so called $H-$ centres depending on the local 
  {\it  Nash-Entropy}   have been shown to hold by Bamler (see Section 6 of \cite{Bam}). 

 \item[(b)] distance estimates and $C^0$ orbifold convergence of $(N^4,g(t))$ 
 as $t \upto   T< \infty,$ in the case that  $n=4$  and $M^4$ is closed   : see Theorem \ref{OrbifoldThmIntro}.  This  generalises the results of Bamler/Zhang \cite{RBQZ,RBQZ2} and Simon  \cite{MSIM1, MSIM2}, where the case that the  scalar curvature remains uniformly bounded on $[0,  T)$ was considered.\\
\item[(c)] if $M$ is closed, and $N=M^4$ is   real four dimensional,   then   the closed   four dimensional  solutions $(M^4,g(t))_{t\in [0,  T)}$  can be extended  for a short time $\si >0$ to
 $[0, T+ \si)$ using the Orbifold Ricci flow, thus generalising  the results of Simon \cite{MSIM2}, where an analogous result was shown
 in the case that the scalar curvature remains uniformly bounded  on $[0,  T).$
\end{itemize} 
We gather some of the main estimates of this paper together in the  following  Theorem.

\begin{thm}\label{mainthmIntro} 
Let $(M,g(t))_{t\in[0, T)}$ with $0<T<\infty$ be a smooth solution to Ricci flow on a closed $n$-dimensional Riemannian manifold. Assume $N  \subseteq M,$ $\Omega$   and $(M^{n},g(t))_{t\in [0, T)}$    are  as in ({\bf B}) and  that $\sup_{t\in [0, T)} \int_{N}|\Sc_{g(t)}|^{\frac{n}{2}+12\alpha}dV_{g(t)} =: L < \infty $   for some constant $0<\alpha<\frac 1 {12}.$  
Then there exists a constant $ \infty > \si_0$   such that 
\begin{eqnarray}\label{nonin1Intro}
  1)  \ \vol_{g(t)}(B_{g(t)}(x,r)) \geq  \si_0 r^n   \ \ \ \ \ \ \  \ \ \ \ \ \ \ \ \ \ \   \ \ \ \ \ \ \ \ \ \ \ &&
\end{eqnarray}
for all $x \in N,$ for all $r\in (0,1).$ 
If we further assume that  $M=M^4$  is real four dimensional, or that  $(M^n,g_0)$   is Kähler (complex  dimension   $\frac n 2$, 
real dimension   $n$) and closed, then we also  have : 
\begin{eqnarray}\label{nonin2Intro} 
&& 2) \ \text { there exists } K_0\in (0,\infty) \cr
&&  \text{ such that }  \int_{N} |\Rm_{g(t)}|_{g(t)}^2 dV_{g(t)}    
  \leq   K_0    \text{ for all }    t \in (0,T), \label{RiemestgoodIntro}\\
&& 3)  \ \text {for all} \     Y>0, \text { there exists } c_1 \text{ such that }  \cr
&&  
  \int_{S}^V \int_{B_{g(s)}(p,Y\sqrt{V-s})} |\Rc_{g(s)}|_{g(s)}^{2+\al^3}  dV_{g(s)}  ds  
  \leq  c_1 (V-S)^{1 + \frac{\al}{16}} \cr
 &&  \text{ for all }  p\in N,  t \in (0,T),   0<S<V \leq  T,\  \text{with} \ Y\sqrt{V-S}\leq 1.
\end{eqnarray}
 
\end{thm} 
\begin{rem}\label{noninflremark}   Theorem 8.1 of   \cite{Bam},  implies   in the setting    ({\rm \bf B})  being considered,   that  a  non-inflating   estimate  holds:  That is there exist  $\infty>\si_1,\si_0>0$ such that 
  1) may be replaced by 
   $$ \ti 1) \ \si_0 r^n \leqslant \vol_{g(t)}(B_{g(t)}(x,r)) \leqslant \si_1r^n $$  for all 
  $x \in N,$ for all $r\in (0,1).$ 
 We   give a   proof of this  non-inflating estimate  in the  setting of the above Theorem (see Section 3), by generalising the proof of Zhang \cite{QSZ1}, who proved non-inflating in a bounded scalar curvature setting. We hope that this method  may be used  in other related settings to obtain related estimates.    
\end{rem}
\begin{rem}
Notice that in case $M^n=M^4$ is real four dimensional, that the inequality 3) has broken the scale: scaling solutions by  constants $C$ larger than one  will result in a similar inequality for the new solution, where the constant $c_1$ on the right hand is replaced by  
$  \frac{c_1}{C^{ \al^3 + \al/16}},$ which goes to zero as  
  $C \to \infty$ (see Section \ref{AnApplication} for an application of this).     
\end{rem}

Using Theorem  \ref{mainthmIntro}, as well as some De-Giorgi-Moser-Nash estimates for Ricci flows in a $\frac{C}{t}$ setting which we present in  Section \ref{DGNM}, we are then able to use  the methods decribed in \cite{MSIM2} (compare \cite{RBQZ, RBQZ2, SHA}) 
 to show the following:
\begin{thm}\label{OrbifoldThmIntro}
Let 
  $(M^4,g(t))_{t \in [0, T)}$ be  a smooth, four dimensional  closed  solution to the  Ricci  flow, $\partt g = - \Ric(g),$  where $N$ and $(M^{n},g(t))_{t\in [0, T)}$  are   as in ({\bf B}), and assume further that  $\int_N |\Sc_{g(t)}|^{2+\si}dV_{g(t)}  \leq L < \infty$  for some $\infty>L,$ $1> \si >0$, for all $t\in [0, T).$  
Then $(N,d(g(t)) ) \to (X,d_X)$ in the Gromov-Hausdorff sense, as $t\upto  T,$ where $(X,d_X)$ is a metric space and a  $C^0$ Riemannian orbifold, and where the convergence and the metric are smooth away from  the, finitely many,  non-manifold points, $x_1, \ldots, x_L$ of $X$.  
If furthermore $N^4=M^4,$ then we can extend the flow with  the Orbifold Ricci flow for a time $\si>0$: there exists a solution $(X,h(t))_{t\in (T,T+ \si]}$ to the Orbifold Ricci flow, such that $(X,d(h(t)) ) \to (X,d_X)$ as $t \downto T$ uniformly.
 \end{thm}

 \section{Structure of the paper and some previous results}
 
 Here we recall the     integral estimates of    \cite{LiuSim1} that  we require in this paper, as well as a slightly modified version of one of the estimates that appear in that paper: 
  \begin{thm}[A minor modification of estimates in  Liu-Simon \cite{LiuSim1} ]\label{LiuSim1}  
  For $n\in\N,$ $1\leq {\rm V} \leq  T < \infty,$  $\alpha\in(0,\frac{1}{12}),$  
let   $(M^{n},g(t))_{t\in [0,T)}$  be a smooth,  real  solution
to Ricci flow and  assume 
that $n=4$ or that $(M^n,g_0)$ is Kähler (complex  dimension $\frac n 2,$ real dimension $n$) and closed.
Assume further that    $N  \subseteq M$   
  is  a  smooth  connected  real $n$-dimensional   submanifold with boundary, and  $N$ and $(M^{n},g(t))_{t\in [0,T)}$   are  as in ({\bf B}) and that 
$$\int_N |\Sc_{g(t)}|^{2+12\alpha} dV_{g(t)}  \leq L< \infty$$ for all $t\in [0,T)$ for some constant $L\in \R^{+}.$ 
Then, for all $l  \in [0,T)$ we have  
\begin{eqnarray} 
&& i)   \int_{N} |\Rm_{g(l)}|_{g(l)}^2 dV_{g(l)}   
  \leq  K_0\label{Riemestgood}, \\\nonumber
   && ii)   \int_{ {\rm V}-2s}^{{\rm V}-s} \int_{N \cap B_{g(t)}(p,r) } |\Rc_{g(t)}|_{g(t)}^4 dV_{g(t)} dt\\
   &&\ \ \   \leq \hat c_1 \frac{1}{s^{1-\alpha} } + 
  \hat c_1 \sup \{     |\Rc_{g(t)}|_{g(t)}^2 \ | \   t \in [{\rm V}-2s,{\rm V}-s] , x    \in B_{g(t)}(p,r) \} s^{1+\alpha}
 \end{eqnarray}
  for all $p \in N,$ $r\in (0,\infty),$   if  $s \in (0,1)$ and ${\rm V}-3s>0$, where  $\hat c_1,K_0$ are constants depending on $ n,$ $N,$ $g(0),$ $\Omega,$ $g|_{\Omega},$  $\frac{1}{\al},$ $T,$ $L.$ 
  If for some $Y,b_1 \in [1,\infty),$ and all $ S\in [V-1,V)$ satisfying  $Y\sqrt{V-S}   \leq   1$    we further assume the weak non-inflating condition    
\begin{eqnarray}
{\rm Vol}_{g(t)}(B_{g(t)}(p, Y\sqrt{V-t} ) ) \leq b_1   |V-t|^2  \cr
\text{ for all  }
   \ t \in [S,V), p \in N   \label{NonExpandTIntro}
\end{eqnarray} 
then we also obtain  
  \begin{eqnarray}
&&iii)  \int_{S}^V \int_{B_{g(t)}(p,Y\sqrt{V-t})} |\Rc_{g(t)}|_{g(t)}^{2+ \alpha^3 }  dV_{g(t)}  dt 
 \leq   \hat c_2    (V-S)^{1 + \frac{\al}{16}} \label{RiccestgoodIntro}
 \end{eqnarray} 
   for all  $S \in [V-1,V),$ satisfying $Y\sqrt{V-S}   \leq   1$  and all $p\in N,$  
  where $\hat c_2< \infty$ depends on  $  b_1, Y, n,N,g(0), \Omega, g|_{\Omega},  \frac{1}{\al}, T, L.$   
   For $V= T,$ we mean  
 $\lim_{V \upto T} \int_{S}^V \int_{B_{g(t)}(p,Y\sqrt{V-t})} |\Rc_{g(t)}|_{g(t)}^{2+\al^{3}}  dV_{g(t)}  dt  
  \leq  \hat c_2   (T-S)^{1 + \frac{\al}{16}} 
 $
  
  \end{thm}
  \begin{proof}
  Claim (i) and (ii)  are the results (i) and (ii) of Theorem 1.3 in \cite{LiuSim1}. 
Claim (iii) is   Equation (2.40) of \cite{LiuSim1} in the case $Y=1,$  and for the general case $ Y\in [1,\infty)$ follows from the proof  of  Equation (2.40) in \cite{LiuSim1}. 
The proof is achieved by:\\
a)   replacing   $ {B_{g(t)}(p,\sqrt{{\rm V}-t})}$ by $ {B_{g(t)}(p,Y\sqrt{{\rm V}-t})} $  everywhere in the proof of (2.40) in \cite{LiuSim1}, and using  ${\rm Vol}_{g(t)}(B_{g(t)}(p, Y\sqrt{V-t} ) ) \leq b_1   |V-t|^2$ instead of   ${\rm Vol}_{g(t)}(B_{g(t)}(p, \sqrt{V-t} ) ) \leq \si_1   |V-t|^2$ 
  at the   steps  of  the proof of (2.40) in \cite{LiuSim1} where the inequality   ${\rm Vol}_{g(t)}(B_{g(t)}(p, \sqrt{V-t} ) ) \leq \si_1   |V-t|^2$ is used,  \\
  b)  replacing the fact that $B_{g(t)}(p,\sqrt{V-t}) \subseteq B_{g(t)}(p,1) \subseteq \Omega \cup N$ for all $t\in [V-1,V]$, shown at the beginning of the proof of   (2.40) in \cite{LiuSim1}, by the fact that 
 $B_{g(t)}(p,Y\sqrt{V-t}) \subseteq B_{g(t)}(p,1) \subseteq \Omega \cup N$  for all $t\in [S,V]$ everywhere in the proof of  (2.40) in \cite{LiuSim1},  
   which follows  from      $Y\sqrt{V-t} \leq 1$  for all $t\in [S,V]$. Here  $\Omega$  is as in condition ({\bf B}). 
  \end{proof}

\begin{rem}\label{remarkext}
If $M^n$ $=M^4$ is real four dimensional, and  we scale the solution by a constant $C>1$, that is  $\hat g(\hat t):= Cg(\frac{\hat t}{C}),$ then we obtain in place of (ii)
\begin{eqnarray*}
  && \int_{ \hat V-2\hat s}^{\hat V-\hat s} \int_{N \cap B_{\hat g(\hat t)}(p,\hat r) } |\Rc_{\hat g(\hat t)} |^4 dV_{\hat g(\hat t)} d\hat t \\
   && = \frac{1}{C}\int_{ V-2 s}^{ V-  s} 
   \int_{N \cap B_{  g(  t)}(p,r)  }|\Rc_{g(t)}|^4 dV_{ g(t)} dt  \\
    &&  \leq   \frac{1}{C}\hat c_2\frac{1}{s^{1-\alpha}}  + 
 \hat c_2 \sup \{     \frac{|\Rc|^2(x,t)}{C^2}  \ | \   t \in [V-2s,V-s] , x    \in B_{g(t)}(p,r)    \} C s^{1+\alpha}  
  \\
 && =   \frac{\hat c_2}{C^{\alpha}}\Big( \frac{1}{{\hat s}^{1-\alpha}}  + 
   \sup \{    |\hat{\Rc}|^2(x,\hat t)  \ | \   \hat t \in [\hat V-2\hat s,\hat V-\hat s] , x    \in B_{\hat g(\hat t)}(p,\hat r)    \}  {\hat s}^{1+\alpha}\Big)
 \end{eqnarray*}
for $\hat s = sC  \in (0, C),$  $\hat r = \sqrt{C}r  \in (0, \sqrt{C}),$and  $\hat V:= V C$. 
 \end{rem}

If $M$ is closed, a very general non-inflating estimate is known to hold due to Bamler ( \cite{Bam}, Theorem 8.1), and this implies that
the weaker condition 
     \eqref{NonExpandTIntro} holds. 
We give a short  reasonably simple  proof of  the   non-inflating  estimate  which is valid  in the  setting.
 we are considering. Our    proof in Section 3, uses  methods, estimates  of Zhang (\cite{QSZ1}), but is otherwise self-contained. We hope that the proof method    may be used  in other related
 settings to obtain related estimates.    
  
At the end of that section, we prove Theorem \ref{mainthmIntro}, as an application of the non-collapsing/non-inflating estimates,  and the integral curvature estimates proved in the prequel paper \cite{LiuSim1} to this paper. 
 In Section \ref{DGNM}, we prove De-Giorgi-Moser-Nash estimates 
to sub-solutions  of a heat like parabolic  equation with a non-negative reaction term, in the setting that the norm of the curvature and the reaction term  are bounded by $\frac{c_0}{t}$.
In the last section, we prove Theorem \ref{OrbifoldThmIntro}, as an application of the estimates proved in the sections before it. In the setting of a closed manifold with 
$\int_M |\Sc_{g(t)}|^{\frac n 2 + \si} dV_{g(t)} \leq L < \infty $ for all $t\in [0,T),$ for some $0<\si<1,$ where $T< \infty$ is a potential blow-up time, we show:\\
 -$(M,g(t)) \to (X,d_X),$ as $t\upto T,$ where $(X,d_X)$ is an orbifold, and that the convergence is   smooth away from the finitely many non-manifold points,\\
 -$(X,d_X)$ can be evolved for a short time using the Orbifold Ricci flow.

In Appendix A, we explain some mostly well known facts on Harmonic coordinates, which are needed in 
the proof of Section \ref{AnApplication},  when examining the behaviour of so called {\it regular points}. 
In Appendix B we prove a version of the {\it Neck  Lemma} of Anderson/Cheeger, which we use to show that small annuli centred at singular points on  the limiting, possibly singular metric space obtained with $t \upto T,$ are connected.

\section{Non-collapsing and non-inflating  estimates}\label{NCNE}

In this section, we consider  smooth   Ricci flows $(M,g(t))_{t\in [0,T)}$ on closed manifolds on  finite time intervals. 
  Ye \cite{RGY} (Theorem 1.1) and Zhang \cite{QSZ, QSZ2} (Theorem  1.1 in \cite{QSZ}  ) proved   a Sobolev inequality in this setting, namely :
\begin{thm}( Ye \cite{RGY}  , Zhang  \cite{QSZ, QSZ2} ) \label{SobolevZhangYe}
Let $(M,g(t))_{t\in [0,T)},$ with $T<\infty$ be a smooth closed solution to Ricci flow.
Then there exists 
$A,B>0$ depending only on \\
i) a lower bound for $\vol(M,g(0)),$  \\
 ii) the Sobolev constant $C_S(M,g_0)$ at time zero, and \\
iii) $T< \infty,$ $n   \in  \N$, $ \inf_M  \Sc_{g(0)} $  \\
such that
\begin{eqnarray} 
\int_M |u|^{\frac {2n}{n-2}} dV_{g(t)} \leq A \int_M ( |\grad u|^2_{g(t)}  + \Sc_{g(t)} \frac{u^2}{4}  )dV_{g(t)}
+ B \int_M  u^2 dV_{g(t)}
\end{eqnarray}
for all smooth $u:M \to \R$. 
\end{thm}
Using this, they also showed ,  Ye,  Theorem 1.7 in \cite{RGY},  and Zhang, Corollary 1 in  \cite{QSZ} (and the errata for this in \cite{QSZ2}),   that if a closed Ricci-Flow solution  $(M,g(t))_{t\in [0,T)}$  satisfies  $|\Sc|(\cdot,t_0)\leq c_0< \infty$ on $B_{g(t_0)}(x_0,r_0)$ 
for some $t_0 \in [0,T)$,  that then {\it  $B_{g(t_0)}(x_0,r_0)$   is $\si_0$ non-collapsed} (with respect  to the metric $g(t_0)$) , where $\si_0 = \si_0(A,B,n) >0$ is a positive constant depending only on $A$ and $B$ and the dimension  $n$: here,  {\it $B_{g(t_0)}(x_0,r_0)$  is  $\si_0$ non-collapsed}    means that $\vol_{g(t_0)}( B_{g(t_0)}(x_0,r)) \geq \si_0 r^n$ for all $0<r\leq \min(r_0,1)$.

Under the assumption  that     $\int_{B_{g(t_0)}(x_0,R_0)} |\Sc_{g(t_0)}|^{\frac n 2}dV_{g(t_0)} \leq \ep_0 $ for some  $t_0\in [0,T)$ for some $R_0>0$ and some $x_0 \in M$,  where $\ep_0 = \ep_0(A,B,n)$ is small enough,  $A$ and $B$   as in Theorem \ref{SobolevZhangYe},  
  Hamanaka,  Lemma $2.5$ in  \cite{SHA},  modified  the arguments of Zhang and Ye,  
  to show that $\vol_{g(t_0)}(B_{g(t_0)}(x_0,r)) \geq \si_0 r^n ,$  for all $r\leq \min(1,R_0)$  where   $\si_0$,  depends only on $A,B,n$.
In the case that $\int_{B_{g(t_0)}(x_0,R_0)} |\Sc_{g(t_0)}|^{\si + \frac n 2}dV_{g(t_0)} \leq L $ for some $L>0, 0<\si<1$ we use a similar argument to those mentioned above to obtain a non-collapsing result.

\begin{lem}\label{noncollapselemma} Let $(M,g)$ be a smooth complete $n$-dimensional Riemannian manifold. Suppose that there exist positive constants $A$ and $B$ such that the Sobolev inequality
\begin{equation}\label{0527Sobolev1}
(\int_M |u|^{\frac{2n}{n-2}} dV_g)^{\frac{n-2}{n}}\leqslant A\int_M ( |\nabla u|^2_g+\frac{\Sc_g}{4}u^2) dV_g+B\int_M u^2 dV_g
\end{equation}
holds true for all $u\in W^{1,2}(M)$. If $\int_{B_g(x,r)}|\Sc_g|^{\frac{n}{2}+\sigma}dV_g\leqslant L$ on a geodesic ball $B_g(x,r)$ with positive constants $r$, $L$ and $\sigma$ satisfying $\frac{A}{4}L^{\frac{2}{n+2\sigma}}r^{\frac{4\sigma}{n+2\sigma}}<1$. Then there holds
\begin{equation}\label{0527noncoll}
\Vol_g(B_g(x,r))\geqslant (\frac{1}{2^{n+3}\tilde{A}+2\tilde{B}r^2})^{\frac{n}{2}}r^n,
\end{equation}
where $\tilde{A}=\frac{A}{1-\frac{A}{4}L^{\frac{2}{n+2\sigma}}r^{\frac{4\sigma}{n+2\sigma}}}$ and $\tilde{B}=\frac{B}{1-\frac{A}{4}L^{\frac{2}{n+2\sigma}}r^{\frac{4\sigma}{n+2\sigma}}}.$
If we further restrict $r$ to 
$r \leq  \min(1,\frac{2^{n+2\si}{4\si}}{A^{\frac{n+2\si}{4\si}} L^{\frac{1}{2\si} } }) $ then we obtain
 \begin{equation}\label{0527noncoll2}
\Vol_g(B_g(x,r))\geqslant (\frac{1}{2^{n+4}A+4B})^{\frac{n}{2}}r^n,
\end{equation}
\end{lem}
\begin{proof}  Assume that $\int_{B_g(x_0,r_0)}|\Sc _g|^{\frac{n}{2}+\sigma}dV_g \leqslant L$ on a closed geodesic ball $B_g(x_0,r_0)$ with positive constants $r_0$, $L$ and $\sigma$ satisfying $\frac{A}{4}L^{\frac{2}{n+2\sigma}}r_0^{\frac{4\sigma}{n+2\sigma}}<1$, but estimate $(\ref{0527noncoll})$ does not hold, that is,
\begin{equation}\label{0527201}
\Vol_g(B_g(x_0,r_0))< \delta_{r_0}r_0^n,
\end{equation}
where $\delta_{r_0}=(\frac{1}{2^{n+3}A_0+2B_0r_0^2})^{\frac{n}{2}}$, $A_0=\frac{A}{1-\frac{A}{4}L^{\frac{2}{n+2\sigma}}r_0^{\frac{4\sigma}{n+2\sigma}}}$ and $B_0=\frac{B}{1-\frac{A}{4}L^{\frac{2}{n+2\sigma}}r_0^{\frac{4\sigma}{n+2\sigma}}}$. 

We next derive a contradiction. Set $\bar{g}=\frac{1}{r_0^2}g$. Then we have for $\bar{g}$,
\begin{equation}\label{0527202}
\Vol_{\bar{g}}(B_{\bar{g}}(x_0,1))< \delta_{r_0}<1
\end{equation}
and 
\begin{equation}\label{0527203}
\int_{B_{\bar{g}}(x_0,1)}|\Sc _{\bar{g}}|^{\frac{n}{2}+\sigma} dV_{\bar{g}}=r_0^{2\sigma}\int_{B_{g}(x_0,r_0)}|\Sc _g|^{\frac{n}{2}+\sigma} dV_{g}\leqslant Lr_0^{2\sigma}
\end{equation}
on $B_{\bar{g}}(x_0,1)$. Moreover, the Sobolev inequality $(\ref{0527Sobolev1})$ for $\bar{g}$ is 
\begin{equation}\label{0527204}
(\int_M |u|^{\frac{2n}{n-2}} dV_{\bar{g}})^{\frac{n-2}{n}}\leqslant A\int_M ( |\nabla u|^2_{\bar{g}}+\frac{\Sc_{\bar{g}}}{4}u^2) dV_{\bar{g}}+Br_0^2\int_M u^2 dV_{\bar{g}}.
\end{equation}
For $u\in C^\infty (M)$ with support contained in $B_{\bar{g}}(x_0,1)$, we then have
\begin{equation*}\label{0527205}
(\int_{B_{\bar{g}}(x_0,1)} |u|^{\frac{2n}{n-2}} dV_{\bar{g}})^{\frac{n-2}{n}}\leqslant A\int_{B_{\bar{g}}(x_0,1)} ( |\nabla u|^2_{\bar{g}}+\frac{\Sc_{\bar{g}}}{4}u^2) dV_{\bar{g}}+Br_0^2\int_{B_{\bar{g}}(x_0,1)} u^2 dV_{\bar{g}}.
\end{equation*}
By using H\"older's inequality and inequality $(\ref{0527203})$, we have
\begin{equation}\label{0527206}
\begin{split}
\int_{B_{\bar{g}}(x_0,1)} \Sc_{\bar{g}}u^2 dV_{\bar{g}}&\leqslant (\int_{B_{\bar{g}}(x_0,1)}|\Sc_{\bar{g}}|^\frac{n+2\sigma}{2} dV_{\bar{g}})^{\frac{2}{n+2\sigma}}(\int_{B_{\bar{g}}(x_0,1)}u^\frac{2n+4\sigma}{n+2\sigma-2} dV_{\bar{g}})^{\frac{n+2\sigma-2}{n+2\sigma}}\\
&\leqslant L^{\frac{2}{n+2\sigma}}r_0^{\frac{4\sigma}{n+2\sigma}}(\int_{B_{\bar{g}}(x_0,1)}u^\frac{2n+4\sigma}{n+2\sigma-2} dV_{\bar{g}})^{\frac{n+2\sigma-2}{n+2\sigma}}\\
&\leqslant L^{\frac{2}{n+2\sigma}}r_0^{\frac{4\sigma}{n+2\sigma}}\left(\left(\int_{B_{\bar{g}}(x_0,1)}u^{2p\frac{n+2\sigma}{n+2\sigma-2}} dV_{\bar{g}}\right)^{\frac{1}{p}}\left(\Vol_{\bar{g}}(B_{\bar{g}}(x_0,1))\right)^{\frac{p-1}{p}}\right)^{\frac{n+2\sigma-2}{n+2\sigma}}\\
&\leqslant L^{\frac{2}{n+2\sigma}}r_0^{\frac{4\sigma}{n+2\sigma}}\left(\int_{B_{\bar{g}}(x_0,1)}u^{\frac{2n}{n-2}} dV_{\bar{g}}\right)^{\frac{n-2}{n}}, 
\end{split}
\end{equation}
where we take $p=\frac{n}{n-2}\cdot\frac{n+2\sigma-2}{n+2\sigma}>1,$ and used 
\eqref{0527202} in the third inequality.  

Then
\begin{equation}\label{05272061}
\begin{split}
(\int_{B_{\bar{g}}(x_0,1)} |u|^{\frac{2n}{n-2}} dV_{\bar{g}})^{\frac{n-2}{n}}&\leqslant \frac{1}{1-\frac{A}{4}L^{\frac{2}{n+2\sigma}}r_0^{\frac{4\sigma}{n+2\sigma}}}\left(A\int_{B_{\bar{g}}(x_0,1)} |\nabla u|^2_{\bar{g}}dV_{\bar{g}}+Br_0^2\int_{B_{\bar{g}}(x_0,1)} u^2 dV_{\bar{g}}\right)\\
&\leqslant A_0\int_{B_{\bar{g}}(x_0,1)} |\nabla u|^2_{\bar{g}}dV_{\bar{g}}+B_0r_0^2\int_{B_{\bar{g}}(x_0,1)} u^2 dV_{\bar{g}}.
\end{split}
\end{equation}
Using H\"older's inequality and inequality \eqref{0527202} again, we have
\begin{equation}\label{05272060}
\begin{split}
\int_{B_{\bar{g}}(x_0,1)} u^2 dV_{\bar{g}}&\leqslant \Vol_{\bar{g}}^{\frac{2}{n}}(B_{\bar{g}}(x_0,1))(\int_{B_{\bar{g}}(x_0,1)}u^\frac{2n}{n-2} dV_{\bar{g}})^{\frac{n-2}{n}}\\
&< \delta_{r_0}^{\frac{2}{n}}(\int_{B_{\bar{g}}(x_0,1)}u^\frac{2n}{n-2} dV_{\bar{g}})^{\frac{n-2}{n}}.
\end{split}
\end{equation}
Hence we deduce 
\begin{equation*}\label{0527207}
\begin{split}
(\int_{B_{\bar{g}}(x_0,1)} |u|^{\frac{2n}{n-2}} dV_{\bar{g}})^{\frac{n-2}{n}}&\leqslant A_0\int_{B_{\bar{g}}(x_0,1)} |\nabla u|^2_{\bar{g}}dV_{\bar{g}}+\frac{B_0r_0^2}{2^{n+3}A_0+2B_0r_0^2}(\int_{B_{\bar{g}}(x_0,1)}u^\frac{2n}{n-2} dV_{\bar{g}})^{\frac{n-2}{n}}\\
&\leqslant A_0\int_{B_{\bar{g}}(x_0,1)} |\nabla u|^2_{\bar{g}}dV_{\bar{g}}+\frac{1}{2}(\int_{B_{\bar{g}}(x_0,1)}u^\frac{2n}{n-2} dV_{\bar{g}})^{\frac{n-2}{n}},
\end{split}
\end{equation*}
which is equivalent to  
\begin{equation}\label{0527208}
(\int_{B_{\bar{g}}(x_0,1)} |u|^{\frac{2n}{n-2}} dV_{\bar{g}})^{\frac{n-2}{n}}\leqslant 2A_0\int_{B_{\bar{g}}(x_0,1)} |\nabla u|^2_{\bar{g}}dV_{\bar{g}}.
\end{equation}
Next we consider an arbitrary domain $\Omega\subset B_{\bar{g}}(x_0,1)$. For $u\in C^\infty(\Omega)$ with support contained in $\Omega$. Denote $\bar{u}=\frac{1}{\int_{\Omega} dV_{\bar{g}}}\int_{\Omega} udV_{\bar{g}}$. Using H\"older inequality and inequality \eqref{0527208}, we have
\begin{equation}\label{0527209}
\begin{split}
\int_{\Omega} |u-\bar{u}|^{2} dV_{\bar{g}}&=\int_{\Omega} u^{2} dV_{\bar{g}}-\frac{1}{\Vol_{\bar{g}}(\Omega)}(\int_{\Omega} u dV_{\bar{g}})^2\\
&\leqslant \int_{\Omega} u^{2} dV_{\bar{g}}\leqslant \Vol^{\frac{2}{n}}_{\bar{g}}(\Omega)(\int_{\bar{\Omega}}u^\frac{2n}{n-2} dV_{\bar{g}})^{\frac{n-2}{n}}\\
&\leqslant 2A_0\Vol^{\frac{2}{n}}_{\bar{g}}(\Omega)\int_{\Omega} |\nabla u|^2_{\bar{g}}dV_{\bar{g}},
\end{split}
\end{equation}
which implies the following Faber-Krahn inequality
\begin{equation}\label{0527210}
\Vol^{\frac{2}{n}}_{\bar{g}}(\Omega)\lambda_1(\Omega)\geqslant\frac{1}{2A_0}.
\end{equation}
Here $\lambda_1(\Omega)$ denotes the first Dirchlet eigenvalue of $-\Delta_{\bar{g}}$ on $\Omega$. Using Lemma \ref{Faber} below, we have
\begin{equation}\label{0527211}
\Vol_{\bar{g}}(B_{\bar{g}}(x,\rho))\geqslant(\frac{1}{2^{n+3}A_0})^{\frac{n}{2}}\rho^n
\end{equation}
for all $B_{\bar{g}}(x,\rho)\subset B_{\bar{g}}(x,1)$. Consequently we have
\begin{equation}\label{0527211second}
(\frac{1}{2^{n+3}A_0+2B_0r_0^2})^{\frac{n}{2}}=\delta_{r_0}>\Vol_{\bar{g}}(B_{\bar{g}}(x,1))\geqslant(\frac{1}{2^{n+3}A_0})^{\frac{n}{2}},
\end{equation}
which is impossible.This finishes the proof.
\end{proof}

In the last step of the above proof, we use the following lemma which was  proved by Carron \cite{Carron}. Ye reproduced Carron's arguments in his  proof of Theorem $6.1$ in \cite{RGY}.
\begin{lem}\label{Faber} (Carron, Proposition $2.4$ in \cite{Carron}) Let $M$ be a smooth complete $n$-dimensional Riemannian manifold and $\Omega\subset M$ be an open bounded regular set. If the Faber-Krahn constant 
\begin{equation}\label{2110000001}
\Lambda_n(\Omega)=\inf\{\Vol^{\frac{2}{n}}_{g}(U)\lambda_1(U)\big|\ U\subset \Omega\ is\ bounded\ open\ set\}
\end{equation}
is positive. Then for any geodesic ball $B_g(x,r)$ contained in $\Omega$, we have
\begin{equation}\label{212}
\Vol_g(B_g(x,r))\geqslant (\frac{\Lambda_n}{2^{n+2}})^{\frac{n}{2}}r^n.
\end{equation}
\end{lem}

Combining Lemma \ref{noncollapselemma} and Theorem \ref{SobolevZhangYe}, we obtain the following. 
\begin{thm}\label{noncollapsethmcom}
Let $(M,g(t))_{t\in [0,  T)}$ be a smooth, closed solution to Ricci flow and $A,B>0$ be the constants from Theorem \ref{SobolevZhangYe}.  Assume furthermore that   
 $\int_{ B_{g(t)}(x,4r) } |\Sc_{g(t)}|^{\frac{n}{2} + \si}  dV_{g(t)} \leq L$ for some constants $0<\si<1$ and $L< \infty$ for all $t\in [0,  T)$ for some $x \in M,$ for some $r  \leq  r_0:= \min(\frac{1}{A^{\frac{n+2\si}{4\si}} L^{\frac{1}{2\si} } }, 1).$ Then we have 
 \begin{equation}\label{noncoll}
\Vol_{g(t)}(B_{g(t)}(x,r))\geqslant (\frac{1}{2^{n+4}A+4B})^{\frac{n}{2}}r^n,
\end{equation}
for all $t\in [0, T).$  
\end{thm}

Qi Zhang also studied the non-inflating case in his paper  \cite{QSZ1}. There  he showed a local non-inflating result in the case that 
the scalar curvature satisfies some local time dependent upper and lower   bounds  :  See Theorem 1.1  in \cite{QSZ1}. 
If ($M,g(t))_{t\in [0,T)}$ is smooth, closed, $T< \infty$  and $ |\Sc(\cdot,t)| \leq c_0< \infty $ for all $t\in [0,T),$ then the result of Qi Zhang shows that there exists a $\si_1>0$ such that $M\times [0,T)$ is $\si_1$  non-inflating, that is $\vol_{g(t)}(B_{g(t)}(x,r)) \leq \si_1r^n$ for all $t\in  [0,T),0<r\leq 1.$
 Hamanaka, also showed (see Lemma 2.2 in \cite{SHA})  that    $M\times [0,T)$ is $\si_1$  non-inflating for some $\si_1>0,$
 if one assumes   an  $L^{\frac  n 2}$ bound on the Scalar curvature,   $\int_M  |\Sc_{g(t)}|^{\frac n 2}dV_{g(t)} \leq E< \infty$ for all $t\in [0,T), $  and a bound of the type 
$\int_0^{T} \sup_{p\in M}  |\Sc_{g(t)}| dt  < \infty. $ 
Later, Bamler (Theorem 8.1 of \cite{Bam}) showed that non-inflating estimates always hold if $M$ is closed.  

\begin{thm}\label{noninthm} (Bamler,  Theorem  8.1 \cite{Bam})
Let $(M,g(t))_{t\in[0,  T)}$ with $  T<\infty$ be a smooth solution to Ricci flow on a closed $n$-dimensional Riemannian manifold, with $\Sc(g(t)) \geq -1$.
Then  $\Vol_{g(t_0)}(B_{g(t_0)}(x,r))\leq  C(n)r^n $ for all $r\in (0,\sqrt{t_0} ),$  
\end{thm}

 We present here a short proof in the setting we are considering, which involves   refining some 
 arguments and estimates   of  Qi Zhang  \cite{QSZ1}  in the setting we are considering, but is otherwise self-contained.  \\
\begin{emptythm} \label{claim} We assume: \\
{\rm 1)   $(M,g(t))_{t\in[0,  T)}$ with $  T<\infty$ is a smooth solution to Ricci flow on a closed $n$-dimensional Riemannian manifold with $\Sc \geq -1$ ,    $Z \subseteq M$ is an open     subset of $M$   with  $\int_{Z}|\Sc_{g(t)}|^{\frac{n}{2}}dV_{g(t)}\leqslant E$  for all $t\in[0, T)$. \\
 2) $B_{g(t_0)}(x,1)$   is $\kappa$ non-collapsed, that is, $\Vol_{g(t_0)}(B_{g(t_0)}(x,r))\geqslant\kappa r^n$ or some $\kappa>0$ for all $r\in (0,1), $ and  $B_{g(t_0)}(x,r) \subseteq Z$.\\
 Then we have
\begin{equation}\label{nonin}
\Vol_{g(t_0)}(B_{g(t_0)}(x,r))\leqslant \Big(1+C_0(1+r^2)^{\frac{n}{2}}\Big)c_1^{-1}J^{-1}(T)e^{2c_2+\frac{2e^{\frac{2B}{n}r^2}E^{\frac{2}{n}}}{3\kappa^{\frac{2}{n}}}}r^n,
\end{equation}
for all $r\in (0,\sqrt{t_0} ),$ 
where $J(s)>0$  is monotone decreasing and defined in $(\ref{2020020001})$ for arbitrary $s>0$,   the constants $c_1,c_2,C_0,$  are those appearing in   $(\ref{202002001})$ and  $(\ref{2020020002}),$   $B$ is the lower bound of the scalar curvature $\Sc_{g(0)},$ and $c_1,c_2,C_0,J^{-1}(T)$  depend only on $(M,g(0))$ and $n$.}
\end{emptythm}
{\it A proof of the non-inflating estimate in this setting  :}\\
The starting point of the proof are the following two heat kernel estimates of  Qi Zhang. Qi Zhang considered the conjugate heat equation coupled with the Ricci flow
\begin{equation}\label{202002002}
\Delta_{g} u-\Sc_g u+ \frac{\partial u}{\partial l}=0, 
\end{equation}
 and proved  ((1.3)   of  \cite{QSZ1})  the following lower bound for 
  $G(z,l;x,t),$ the fundamental solution of this equation, 
\begin{equation}\label{202002001}
G(z,l;x,t)\geqslant\frac{c_1J(t)}{(t-l)^{\frac{n}{2}}}e^{-2c_2\frac{d_t^2(z,x)}{t-l}}e^{-\frac{1}{\sqrt{t-l}}\int_l^t\sqrt{t-s}\Sc_{g(s)}ds},
\end{equation}
where   
\begin{equation}\label{2020020001}
J(s)=e^{-\alpha-s\beta-s\sup\limits_M \Sc^-_{g(0)}}.
\end{equation}
The constants $\alpha$, $\beta$, $c_1$ and $c_2$ depend only on the dimension $n$, the coefficients in the Sobolev inequality for $g(0)$ and the infimum of Perelman's $\mathcal{F}$ entropy for $g(0)$.  Qi Zhang    ((1.2) in \cite{QSZ1}) also obtained an upper bound for   $G$ of the following form
\begin{equation}\label{2020020002}
\int_MG(z,l;x,t)dV_{g(t)}\leqslant1+C_0(1+t-l)^{\frac{n}{2}},
\end{equation}
where the constant $C_0$ depends only on a lower bound of $\Sc_{g(0)}$ and $n$. 
We take  $z=x_0$, $t=t_0$ and $l=t_0-r^2$ with $r\in(0,\sqrt{t_0})$ in $(\ref{202002001})$. For $x$ such that $d_{t_0}(x,x_0)\leqslant r$, we have
\begin{equation*}\label{202002008}
\begin{split}
G(x_0,t_0-r^2;x,t_0)&\geqslant\frac{c_1J(t_0)}{r^n}e^{-2c_2\frac{d^2_{t_0}(x,x_0)}{r^2}}e^{-\frac{1}{r}\int_{t_0-r^2}^{t_0}\sqrt{t_0-s}\Sc_{g(s)}ds}\\
&\geqslant \frac{c_1J(t_0)}{r^n}e^{-2c_2}e^{-\frac{1}{r}\int_{t_0-r^2}^{t_0}\sqrt{t_0-s}\Sc_{g(s)}ds}.
\end{split}
\end{equation*}
Integrating this  inequality over $M$, and  using $(\ref{2020020002})$, we get 
\begin{equation}\label{202002009}
\begin{split}
&\ \ \ 1+C_0(1+r^2)^{\frac{n}{2}}\geqslant\int_MG(x_0,t_0-r^2;x,t_0)dV_{g(t_0)}\\
&\geqslant\int_{B_{g(t_0)}(x_0,r)}G(x_0,t_0-r^2;x,t_0)dV_{g(t_0)}\\
&\geqslant \frac{c_1J(t_0)}{r^n}e^{-2c_2}\int_{B_{g(t_0)}(x_0,r)}e^{-\frac{1}{r}\int_{t_0-r^2}^{t_0}\sqrt{t_0-s}\Sc_{g(s)}ds}dV_{g(t_0)}\\
&\geqslant \frac{c_1J(t_0)}{r^n}e^{-2c_2}\Vol_{g(t_0)}(B_{g(t_0)}(x_0,r))e^{-\frac{1}{r\Vol_{g(t_0)}(B_{g(t_0)}(x_0,r))}\int_{B_{g(t_0)}(x_0,r)}\int_{t_0-r^2}^{t_0}\sqrt{t_0-s}\Sc_{g(s)}ds\ dV_{g(t_0)}},
\end{split}
\end{equation}
where we  used Jensen's inequality in the last step.
Note that $\Sc_{	g(t)} \geq -  B$ implies $\int_\Omega f(x)  dV_{g(t)}(x)\leq e^{B(t-s)}\int_\Omega f(x) dV_{g(s)}(x)$ for $0\leq s \leq t<T$ for  any open set $\Omega \subseteq M$ and any non-negative integrable (in space ) function $f,$ as one sees by differentiating  $\int_\Omega f(x)  dV_{g(t)}(x)$ and integrating with respect to time. 
 Using this fact and the   H\"older inequality, we deduce 
\begin{equation*}\label{2020020010}
\begin{split}
&\ \ -\frac{1}{r\Vol_{g(t_0)}(B_{g(t_0)}(x_0,r))}\int_{B_{g(t_0)}(x_0,r)}\int_{t_0-r^2}^{t_0}\sqrt{t_0-s}\Sc_{g(s)}ds\ dV_{g(t_0)}\\
&\geqslant-\frac{1}{r\Vol_{g(t_0)}(B_{g(t_0)}(x_0,r))}\int_{t_0-r^2}^{t_0}\sqrt{t_0-s}\Big(\int_{B_{g(t_0)}(x_0,r)}|\Sc_{g(s)}|^{\frac{n}{2}}dV_{g(t_0)}\Big)^{\frac{2}{n}}\Vol^{\frac{n-2}{n}}_{g(t_0)}(B_{g(t_0)}(x_0,r))ds\\
&\geqslant -\frac{1}{r\Vol^{\frac{2}{n}}_{g(t_0)}(B_{g(t_0)}(x_0,r))}\int_{t_0-r^2}^{t_0}\sqrt{t_0-s}\Big(\int_{Z}|\Sc_{g(s)}|^{\frac{n}{2}}dV_{g(t_0)}\Big)^{\frac{2}{n}}ds\\
&\geqslant -\frac{1}{r\Vol^{\frac{2}{n}}_{g(t_0)}(B_{g(t_0)}(x_0,r))}\int_{t_0-r^2}^{t_0}\sqrt{t_0-s}\Big(e^{B(t_0-s)}\int_{Z}|\Sc_{g(s)}|^{\frac{n}{2}}dV_{g(s)}\Big)^{\frac{2}{n}}ds\\
&\geqslant -\frac{1}{r\Vol^{\frac{2}{n}}_{g(t_0)}(B_{g(t_0)}(x_0,r))}\int_{t_0-r^2}^{t_0}\sqrt{t_0-s}\Big(e^{Br^2}\int_{Z}|\Sc_{g(s)}|^{\frac{n}{2}}dV_{g(s)}\Big)^{\frac{2}{n}}ds\\
&\geqslant-\frac{e^{\frac{2B}{n}r^2}E^{\frac{2}{n}}}{r\Vol^{\frac{2}{n}}_{g(t_0)}(B_{g(t_0)}(x_0,r))}\int_{t_0-r^2}^{t_0}\sqrt{t_0-s}\ ds\\
&=-\frac{2e^{\frac{2B}{n}r^2}E^{\frac{2}{n}}}{3\Vol^{\frac{2}{n}}_{g(t_0)}(B_{g(t_0)}(x_0,r))}r^2\geqslant-\frac{2e^{\frac{2B}{n}r^2}E^{\frac{2}{n}}}{3\kappa^{\frac{2}{n}} }.
\end{split} 
\end{equation*}
Inserting this inequality into \eqref{202002009}, we see that 
\begin{equation}\label{2020020011}
\Vol_{g(t_0)}(B_{g(t_0)}(x_0,r))\leqslant \Big(1+C_0(1+r^2)^{\frac{n}{2}}\Big)c_1^{-1}J^{-1}(t_0)e^{2c_2+\frac{2e^{\frac{2B}{n}r^2}E^{\frac{2}{n}}}{3\kappa^{\frac{2}{n}}}}r^n,
\end{equation}
which is the claimed statement. 
$\Box$ ({\it End of proof of the non-inflating estimate  }).

Combining  Theorem \ref{noncollapsethmcom}, with the non-inflating result of Bamler  and  Theorem \ref{LiuSim1}, we can now prove Theorem \ref{mainthmIntro}.\\
{\it Proof of Theorem \ref{mainthmIntro} and $\ti 1)$ of Remark \ref{noninflremark} } 

  For $Z:= N \cup \Omega$, where $\Omega$ is as in ({\bf B}), we have $\int_{Z}|\Sc_{g(t)}|^{\frac{n}{2}+\si}dV_{g(t)}\leqslant  L_0$  where  $L_0$
depends on $N,\Omega, g|_{\Omega}, g_0,L.$  
Due to the definition of $\Omega$ we have $ B_{g(t)}(x,r) \subseteq Z$ for all $ x \in N$ all $r \leq  10,$  for all $t\in [0, T).$   Hence      
the assumptions  of Theorem \ref{noncollapsethmcom} are satisfied. Hence $\vol_{g(t)}(B_{g(t)} (x,r)) \geq \si_0r^n$ for all $ x \in N$ for all $t \in [0, T)$  , for all $r\leq 1$ where $\si_0= \si_0(N,\Omega, g|_{\Omega}, g_0,L).$ 

We know  $\Sc_{g(t)} \geq -1$ for all $t\in [0,T)$ from the  assumption  in ({\bf B}) , and hence $\partt \vol_{g(t)}(M) = \partt (\int_M dV_{g(t))})  = -\int_M \Sc_{g(t)}dV_{g(t)} \leq  \vol_{g(t)} (M)$.
Hence $\vol_{g(t)}(M) \leq e^{ T}\vol_{g(0)}(M) = K_0,$ and consequently 
$\int_Z |\Sc_{g(t)}|^{\frac n 2} dV_{g(t)} \leq (\int_Z |\Sc_{g(t)}|^{\si +\frac n 2}dV_{g(t)})^{\frac{n}{n+2\si}} 
(\vol_{g(t)}(M) )^{\frac{2\si}{n+2\si}}=:\ti C.$
The non-inflating result in  $\ti 1)$ now  follows from Bamler (\cite{Bam} Theorem 8.1), or alternatively:   
using the fact that $ B_{g(t)}(x,r) \subseteq Z$ for all $ x \in N$ all $r \leq 1$  for all $t\in [0, T),$ we see the assumptions   \eqref{claim}  of the proof above are satisfied for  each $x\in N,$ $t=t_0 \in [0,T)$, which also 
leads to   $\ti 1)$. 
 The estimate 2) follows immediately from i) of Theorem \ref{LiuSim1}. 
The  non-inflating condition   implies the weaker condition   
  $ \vol(B_{g(t)}(x,Y\sqrt{V-t})) \leq \si_1 Y^n|V-t|^2 $ for all $t \in [S,V],$ $Y\sqrt{V-S}\leq 1$.
Hence,    
(iii)  from 
 Theorem \ref{LiuSim1}, combined with {(\bf B)} implies  3), the last claimed result. 
$\Box.$
\medskip

\section{Local estimates for heat flow coupled with Ricci flow in a c/t setting} \label{DGNM}

In this section, we give a local estimate for the non-negative solution of the linear heat equation associated with the Ricci flow by using iteration.

Let $(M,g(t))_{t\in[0,T)}$ with $T<\infty$ be a smooth solution to Ricci flow on an $n$-dimensional Riemannian manifold. Assume that $B_{g(t)}(x,r)\subset\subset M^n$ for all $0\leqslant t<T$. We further assume that there exists a constant $c_0$ such that
\begin{equation}\label{0427001}
|\Rm_{g(t)}|_{g(t)}\leqslant\frac{c_0}{t}\ \ \ \ on\ B_{g(t)}(x,r)\ \ \ for\ all\ t\in(0,T).
\end{equation}
Then for any $\vare\in(0,1)$ and $0<r_1<r_2$, by using the estimates and methods of Perelman \cite{GP}, one can construct cut-off function $\phi=\phi(\cdot,t)$ 
 (see Lemma $7.1$ in \cite{ST16},  also $(4.21)$ in \cite{MSIM2}), with the following properties.   There exists a positive constant $\widehat T$ such that on $M\times\left( [0,\widehat T)\cap [0,T) \right)$, there holds
\begin{eqnarray}\label{cut}
&& \frac{\partial\phi}{\partial t}\leqslant \Delta_{g(t)}\phi,\cr
&& |\nabla\phi|_{g(t)}\leqslant\frac{\alpha}{\vare(r_2-r_1)},\cr
&& \phi=e^{-\sigma t}\ \ \ \ \ on\ \ \ B_{g(t)}(x,r_1),\cr
&& \phi=0\ \ \ \ \ \ \ \ \ on\ \ \ B^c_{g(t)}(x,r_2),
\end{eqnarray}
where $\widehat T=\min\{r_1^2, \alpha c_1(r_2-r_1)^2\}$, $\sigma=\frac{\alpha}{\vare(r_2-r_1)^2}$, $c_1>0$ is a constant depending only on $c_0$ and $n$, and $\alpha>0$ is a universal constant.   

We study the heat equation
\begin{equation}\label{0322002}
\frac{\partial f(t)}{\partial t}\leqslant \Delta_{g(t)} f(t)+\curlyL(t)f(t)
\end{equation}
with $0\leqslant t\leqslant T$, where $f$ and $\curlyL$ are non-negative functions defined 
on  $ \cup_{t\in (0,T)} B_{g(t)}(x,r) \times \{t\}$ and
\begin{eqnarray}\label{0905001}
&& \curlyL(t)\leqslant \frac{\ell}{t}\ \ \ \text{on}\ \ \ B_{g(t)}(x,r)\  \text{for all }\ t \in (0,T) \cr
&&  \text{ for some positive constant}  \ \  \ell\geq 1 .
\end{eqnarray}
  We assume that the     there exist constants $A>0$ and $B\geq 1$ such that 
  for any $t\in[0,T],$  $B_{g(t)}(x,r) \subseteq M $,  and $h\in C^\infty_0(B_{g(t)}(x,r)),$  that there holds
\begin{equation}\label{0322003}
\Big(\int_{B_{g(t)}(x,r)} |h|^{\frac{2n}{n-2}} dV_{g(t)}\Big)^{\frac{n-2}{n}}\leqslant A\int_{B_{g(t)}(x,r)} \Big(4|\nabla h|^2_{g(t)}+\big(\Sc_{g(t)}+B\big)h^2\Big) dV_{g(t)}.
\end{equation}
If   $M$ is  a closed manifold then this always is the case, due to Theorem \ref{SobolevZhangYe} of Ye and Zhang. 
\begin{thm}\label{MoserIteration}
Let $(M,g(t))_{t\in[0,T)}$ with $T<\infty$ be a smooth solution to Ricci flow on an $n$-dimensional Riemannian manifold and $f, \curlyL$ are   non-negative smooth functions defined on $ \cup_{t\in (0,T)} B_{g(t)}(x,r) \times \{t\} $ satisfying
  \eqref{0322002} and  \eqref{0905001}. 
Assume that
\eqref{0322003} holds,  that  $\Sc_{g(t)}  \geq -B$ for all $t \in [0,T),$ where $B\geq 1,$ that 
 $B_{g(t)}(x,r)  \subset \subset M^n$ for all $0\leqslant t<T,$ $ \cup_{t\in (0,T)} B_{g(t)}(x,r) \subset \subset M$  
    and that there 
       exists a constant $c_0$ such that
\begin{equation} 
|\Rm_{g(t)}|_{g(t)}\leqslant\frac{c_0}{t}\ \ \ \ \ on\ \ B_{g(t)}(x,r)\ \ \ for\ all\ t\in(0,T).
\end{equation}
Then for any $p\geqslant2$, there exist positive constants $K_1$ depending only on $n$, $p$, $A$ and $c_0$, and $\widehat T$ depending only on $r$, $c_0$ and $n$ such that for any $t\in(0,\widehat T)\cap(0,T)$, we have
\begin{equation}
\sup\limits_{B_{g(t)}(x,\frac{r}{2})}f(t)\leqslant K_1\Big(B+\frac{2n\ell p}{t}+\frac{32\alpha^2n^2}{r^2}\Big)^{\frac{n+2}{2p}}\Big(\int_{\frac{t}{2}}^{t}\int_{B_{g(s)}(x,r)}f^{p}dV_{g(s)}ds\Big)^{\frac{1}{p}}.
\end{equation}
For any $0<p<2$, there exists a positive constant $K_2$ depending only on $n$, $p$, $A$ and $c_0$ such that for any $t\in (0,\widehat T)\cap(0,T)$, we have
\begin{equation}
\sup\limits_{B_{g(t)}(x,\frac{r}{2})}f(t)\leqslant K_2\Big(B+\frac{4n \ell}{t}+\frac{32\alpha^2n^2}{r^2}\Big)^{\frac{n+2}{2p}}\Big(\int_{\frac{t}{2}}^{t}\int_{B_{g(s)}(x,r)}f^{p}dV_{g(s)}ds\Big)^{\frac{1}{p}}.
\end{equation}
More precisely, $K_1=(4A)^{\frac{n}{2p}}(1+\frac{2}{n})^{\frac{(n+2)^2}{2p}}e^{\frac{2c_1\alpha^2 (n+2)}{p}}$, $K_2=\frac{\gamma^{\frac{n+2}{p}}p(2-p)^{\frac{2-p}{p}}(4A)^{\frac{n}{2p}}(1+\frac{2}{n})^{\frac{(n+2)^2}{2p}}e^{\frac{2c_1\alpha^2(n+2)}{p}}}{(2\gamma^{\frac{n+2}{p}}-1)(1-\gamma)^{\frac{n+2}{p}}}$ with constant $\gamma\in(0,1)$ satisfying $2\gamma^{\frac{n+2}{p}}>1$, $\widehat T=\min\{\frac{r^2}{4}, \frac{\alpha c_1r^2}{4}\}$,  $A>0$ and  $B\geq 1$ come from \eqref{0322003} and $B$ satisfies $\Sc_{g(t)}  \geq -B$ for all $t \in [0,T),$ $\ell$ comes from \eqref{0905001}, $c_1$ is a positive constant depending only on $c_0$ and $n$, and $\alpha$ is a universal constant coming form \eqref{cut}.
\end{thm}
\begin{proof}
We first consider the case $p\geqslant2$. Let $\phi$ be a smooth function compactly supported in $B_{g(t)}(x,r)$ and satisfy heat inequality 
\begin{equation}
\frac{\partial\phi}{\partial t}\leqslant \Delta_{g(t)}\phi.
\end{equation}
For any $p\geqslant2$, direct computations show that
\begin{equation}\label{0322006}
\begin{split}
\frac{d}{dt}\int_M\phi^2f^pdV_{g(t)}&= \int_M2\phi\frac{\partial\phi}{\partial t}f^{p}dV_{g(t)}+p\int_M\phi^2f^{p-1}\frac{\partial f}{\partial t}dV_{g(t)}\\
&\ \ \ -\int_M\phi^2f^p\Sc_{g(t)}dV_{g(t)}\\
&\leqslant \int_M2\phi\Delta_{g(t)}\phi f^{p}dV_{g(t)}+p\int_M\phi^2f^{p-1}(\Delta_{g(t)}f+\curlyL f)dV_{g(t)}\\
&\ \ \ -\int_M\phi^2f^p\Sc_{g(t)}dV_{g(t)}.
\end{split}
\end{equation}
Applying integration by parts and Young's inequality, we get
\begin{equation*}\label{0322007}
\begin{split}
&\ \ \ \int_M2\phi\Delta_{g(t)}\phi f^{p}dV_{g(t)}+p\int_M\phi^2f^{p-1}\Delta_{g(t)}fdV_{g(t)}\\
&=-p(p-1)\int_M\phi^2f^{p-2}|\nabla f|_{g(t)}^2dV_{g(t)}-2\int_Mf^{p}|\nabla\phi|^2_{g(t)}dV_{g(t)}-4p\int_M\phi f^{p-1}(\nabla \phi,\nabla f)_{g(t)}dV_{g(t)}\\
&=-\frac{4(p-1)}{p}\int_M\Big(|\nabla(\phi f^{\frac{p}{2}})|^2_{g(t)}-f^p|\nabla\phi|^2_{g(t)}-p\phi f^{p-1}(\nabla \phi,\nabla f)_{g(t)}\Big)dV_{g(t)}\\
&\ \ \ \ -2\int_Mf^{p}|\nabla\phi|^2_{g(t)}dV_{g(t)}-4p\int_M\phi f^{p-1}(\nabla \phi,\nabla f)_{g(t)}dV_{g(t)}\\
&=-\frac{4(p-1)}{p}\int_M|\nabla(\phi f^{\frac{p}{2}})|^2_{g(t)}dV_{g(t)}+\frac{2(p-2)}{p}\int_Mf^p|\nabla\phi|^2_{g(t)}dV_{g(t)}-4\int_M\phi f^{p-1}(\nabla \phi,\nabla f)_{g(t)}dV_{g(t)}\\
&=-\frac{4(p-1)}{p}\int_M|\nabla(\phi f^{\frac{p}{2}})|^2_{g(t)}dV_{g(t)}+\frac{2(p+2)}{p}\int_Mf^p|\nabla\phi|^2_{g(t)}dV_{g(t)}-\frac{8}{p}\int_Mf^{\frac{p}{2}}\big(\nabla \phi,\nabla (\phi f^{\frac{p}{2}})\big)_{g(t)}dV_{g(t)}\\
&\leqslant-\frac{4(p-1)}{p}\int_M|\nabla(\phi f^{\frac{p}{2}})|^2_{g(t)}dV_{g(t)}+\frac{2(p+2)}{p}\int_Mf^p|\nabla\phi|^2_{g(t)}dV_{g(t)}\\
&\ \ \ \ +\frac{2}{p}\int_M|\nabla(\phi f^{\frac{p}{2}})|^2_{g(t)}dV_{g(t)}+\frac{8}{p}\int_Mf^p|\nabla\phi|^2_{g(t)}dV_{g(t)}\\
&=\frac{-p-3(p-2)}{p}\int_M|\nabla(\phi f^{\frac{p}{2}})|^2_{g(t)}dV_{g(t)}+(2+\frac{12}{p})\int_Mf^p|\nabla\phi|^2_{g(t)}dV_{g(t)}\\
&\leqslant-\int_M|\nabla(\phi f^{\frac{p}{2}})|^2_{g(t)}dV_{g(t)}+8\int_Mf^p|\nabla\phi|^2_{g(t)}dV_{g(t)},
\end{split}
\end{equation*}
in view of the fact that we are assuming $p\geq 2$.
Inserting the inequality $(\ref{0322006})$ into the one we just obtained, we see 
\begin{equation}\label{0427003}
\begin{split}
&\ \ \ \ \frac{d}{dt}\int_M\phi^2f^pdV_{g(t)}+\int_M\phi^2f^p\Sc_{g(t)}dV_{g(t)}+\int_M|\nabla(\phi f^{\frac{p}{2}})|^2_{g(t)}dV_{g(t)}\\
&\leqslant p\int_M\curlyL\phi^2f^{p}dV_{g(t)}+8\int_Mf^p|\nabla\phi|^2_{g(t)}dV_{g(t)}.
\end{split}
\end{equation}
and then
\begin{equation}\label{0322008}
\begin{split}
&\ \ \ \ \frac{d}{dt}\int_M\phi^2f^pdV_{g(t)}+\frac{1}{4}\int_M\phi^2f^p(\Sc_{g(t)}+B)dV_{g(t)}+\int_M|\nabla(\phi f^{\frac{p}{2}})|^2_{g(t)}dV_{g(t)}\\
&\leqslant\frac{d}{dt}\int_M\phi^2f^pdV_{g(t)}+\int_M\phi^2f^p(\Sc_{g(t)}+B)dV_{g(t)}+\int_M|\nabla(\phi f^{\frac{p}{2}})|^2_{g(t)}dV_{g(t)}\\
&\leqslant p\int_M\curlyL\phi^2f^{p}dV_{g(t)}+B\int_M\phi^2f^{p}dV_{g(t)}+8\int_M|\nabla\phi|^2_{g(t)}f^pdV_{g(t)}.
\end{split}
\end{equation}
Now for  any $0<\tau<\tau'\leqslant T'<\widetilde T\leqslant T$, we let
$$\psi=\left\{
\begin{aligned}
 &0,\ \ \ \ \ \ \ \ \ \ \ 0\leqslant t\leqslant\tau,\\
 &\frac{t-\tau}{\tau'-\tau}\ \ \ \ \ \tau\leqslant t\leqslant\tau',\\
 &1,\ \ \ \ \ \ \ \ \ \ \ \tau'\leqslant t< \widetilde T.
\end{aligned}
\right.$$
Multiplying $(\ref{0322008})$ by $\psi$, we have
\begin{equation}\label{0322009}
\begin{split}
&\ \ \ \ \frac{d}{dt}\Big(\psi\int_M\phi^2f^pdV_{g(t)}\Big)+\frac{1}{4}\psi\int_M\phi^2f^p(\Sc_{g(t)}+B)dV_{g(t)}+\psi\int_M|\nabla(\phi f^{\frac{p}{2}})|^2_{g(t)}dV_{g(t)}\\
&\leqslant p\psi\int_M\curlyL\phi^2f^{p}dV_{g(t)}+(B\psi+\psi')\int_M\phi^2f^{p}dV_{g(t)}+8\psi\int_Mf^p|\nabla\phi|^2_{g(t)}dV_{g(t)}\\
&\leqslant p\int_M\curlyL\phi^2f^{p}dV_{g(t)}+(B+\frac{1}{\tau'-\tau})\int_M\phi^2f^{p}dV_{g(t)}+8\int_Mf^p|\nabla\phi|^2_{g(t)}dV_{g(t)}.
\end{split}
\end{equation}
Integrating $(\ref{0322009})$ from $\tau$ to $t$ gives
\begin{equation}\label{03220010}
\begin{split}
&\ \ \ \ \psi\int_M\phi^2f^pdV_{g(t)}+\frac{1}{4}\int_\tau^t\int_M\psi\Big(\phi^2f^p(\Sc_{g(s)}+B)+4|\nabla(\phi f^{\frac{p}{2}})|^2_{g(s)}\Big)dV_{g(s)}ds\\
&\leqslant p\int_\tau^t\int_M\curlyL\phi^2f^{p}dV_{g(s)}ds+(B+\frac{1}{\tau'-\tau})\int_\tau^t\int_M\phi^2f^{p}dV_{g(s)}ds\\
&\ \ \ +8\int_\tau^t\int_Mf^p|\nabla\phi|^2_{g(s)}dV_{g(s)}ds.
\end{split}
\end{equation}
Therefore, for any $\tau<\tau'\leqslant t\leqslant T'< \widetilde T\leqslant T$, we have
\begin{equation}\label{03220011}
\begin{split}
&\ \ \ \ \int_M\phi^2f^pdV_{g(t)}+\frac{1}{4}\int_{\tau'}^t\int_M\Big(\phi^2f^p(\Sc_{g(s)}+B)+4|\nabla(\phi f^{\frac{p}{2}})|^2_{g(s)}\Big)dV_{g(s)}ds\\
&\leqslant p\int_\tau^{T'}\int_M\curlyL\phi^2f^{p}dV_{g(t)}dt+(B+\frac{1}{\tau'-\tau})\int_\tau^{T'}\int_M\phi^2f^{p}dV_{g(t)}dt\\
&\ \ \ +8\int_\tau^{T'}\int_Mf^p|\nabla\phi|^2_{g(t)}dV_{g(t)}dt.
\end{split}
\end{equation}
Therefore, we have
\begin{equation}\label{03220012}
\begin{split}
&\ \ \ \ \sup\limits_{t\in[\tau',T']}\int_M\phi^2f^pdV_{g(t)}+\frac{1}{4}\int_{\tau'}^{T'}\int_M\Big(\phi^2f^p(\Sc_{g(t)}+B)+4|\nabla(\phi f^{\frac{p}{2}})|^2_{g(t)}\Big)dV_{g(t)}dt\\
&\leqslant p\int_\tau^{T'}\int_M\curlyL\phi^2f^{p}dV_{g(t)}dt+(B+\frac{1}{\tau'-\tau})\int_\tau^{T'}\int_M\phi^2f^{p}dV_{g(t)}dt\\
&\ \ \ +8\int_\tau^{T'}\int_Mf^p|\nabla\phi|^2_{g(t)}dV_{g(t)}dt.
\end{split}
\end{equation}

For any $0<r'<r''\leqslant r$, we choose a cut-off function $\phi$ satisfying $(\ref{cut})$ with $r_1=r'$, $r_2=r''$ and $\vare=\frac{1}{2}$. Then for $0<\tau<\tau'\leqslant T'<\widetilde T$ with $\widetilde T=\min\{\widehat T, T\}=\min\{{r'}^{2}, c_1\alpha(r''-r')^2, T\}$, where positive constants $c_1$ and $\alpha>0$ come from \eqref{cut}, by using \eqref{03220012}, H\"older inequality and Sobolev inequality $(\ref{0322003})$, we have
\begin{equation}\label{03220015}
\begin{split}
&\ \ \ \ \int_{\tau'}^{T'}\int_{B_{g(t)}(x,r')}e^{-2\sigma t(1+\frac{2}{n})}f^{p(1+\frac{2}{n})}dV_{g(t)}dt\\
&\leqslant\int_{\tau'}^{T'}\Big(\int_{B_{g(t)}(x,r')}e^{-2\sigma t}f^{p}dV_{g(t)}\Big)^{\frac{2}{n}}\Big(\int_{B_{g(t)}(x,r')}e^{-2\sigma t\frac{n}{n-2}}f^{p\frac{n}{n-2}}dV_{g(t)}\Big)^{\frac{n-2}{n}}dt\\
&\leqslant\int_{\tau'}^{T'}\Big(\int_{M}\phi^2f^{p}dV_{g(t)}\Big)^{\frac{2}{n}}\Big(\int_{M}(\phi f^{\frac{p}{2}})^{\frac{2n}{n-2}}dV_{g(t)}\Big)^{\frac{n-2}{n}}dt\\
&\leqslant\sup\limits_{t\in[\tau',T']}\Big(\int_{M}\phi^2f^{p}dV_{g(t)}\Big)^{\frac{2}{n}}\int_{\tau'}^{T'}\Big(\int_{M}(\phi f^{\frac{p}{2}})^{\frac{2n}{n-2}}dV_{g(t)}\Big)^{\frac{n-2}{n}}dt\\
&\leqslant\sup\limits_{t\in[\tau',T']}\Big(\int_{M}\phi^2f^{p}dV_{g(t)}\Big)^{\frac{2}{n}}A\int_{\tau'}^{T'}\int_{M} \Big(4|\nabla (\phi f^{\frac{p}{2}})|^2_{g(t)}+\big(\Sc_{g(t)}+B\big)\phi^2 f^{p}\Big) dV_{g(t)}dt\\
&\leqslant 4A\Big[(B+\frac{1}{\tau'-\tau})\int_\tau^{T'}\int_M\phi^2f^{p}dV_{g(t)}dt+p\int_\tau^{T'}\int_M\curlyL\phi^2f^{p}dV_{g(t)}dt\\
& \ \ \ \ \ \ \ \ \ +8\int_\tau^{T'}\int_Mf^p|\nabla\phi|^2_{g(t)}dV_{g(t)}dt\Big]^{1+\frac{2}{n}}\\
&\leqslant 4A\Big(\frac{\ell p}{\tau}+B+\frac{32\alpha^2}{(r''-r')^2}+\frac{1}{\tau'-\tau}\Big)^{1+\frac{2}{n}}\Big(\int_\tau^{T'}\int_{B_{g(t)}(x,r'')}f^{p}dV_{g(t)}dt\Big)^{1+\frac{2}{n}}.
\end{split}
\end{equation}

For all $p\geqslant2$ and $0<s<T'$, we define
\begin{equation}\label{03220016}
H(p,s,\rho)=\Big(\int_s^{T'}\int_{B_{g(t)}(x,\rho)}f^{p}dV_{g(t)}dt\Big)^{\frac{1}{p}}.
\end{equation}
For any $0<\tau<\tau'\leqslant T'<\widetilde T$ and $0<r'<r''\leqslant r$, by $(\ref{03220015})$, we have
\begin{equation}\label{03220017}
H\Big(p(1+\frac{2}{n}),\tau',r'\Big)\leqslant \big(4Ae^{2\sigma \widetilde T(1+\frac{2}{n})}\big)^{\frac{1}{p(1+\frac{2}{n})}}\Big(\frac{\ell p}{\tau}+B+\frac{32\alpha^2}{(r''-r')^2}+\frac{1}{\tau'-\tau}\Big)^{\frac{1}{p}}H(p,\tau,r'').
\end{equation}
For any $0<t'<t''< T'<\widetilde T$, we let $\nu=1+\frac{2}{n}$, $p_k=p_0\nu^k$, $\tau_k=t'+(1-\nu^{-k})(t''-t')$ and $r_k=r''-(1-\nu^{-k})(r''-r')$. Then we have
\begin{equation}\label{03220018}
\frac{p_k}{\tau_k}\leqslant\frac{np_0\nu^{2k}}{2(t''-t')},\ \ \ \frac{1}{\tau_{k+1}-\tau_k}=\frac{n\nu^{k+1}}{2(t''-t')}\ \ \ and\ \ \ \frac{1}{(r_k-r_{k+1})^2}=\frac{n^2\nu^{2(k+1)}}{4(r''-r')^2}.
\end{equation}
Repeating $(\ref{03220017})$ for $p_k$, $\tau_k$ and $r_k$ with $k\in\mathbb{N}$, we have
\begin{equation*}\label{03220019}
\begin{split}
&\ \ \ \ H(p_{k+1},\tau_{k+1},r_{k+1})\\
&\leqslant\big(4Ae^{2\sigma \widetilde T(1+\frac{2}{n})}\big)^{\frac{1}{p_{k+1}}}\Big(\frac{\ell p_k}{\tau_k}+\frac{B}{p_0}p_k+\frac{8\alpha^2n^2\nu^{2(k+1)}}{(r''-r')^2}+\frac{n\nu^{k+1}}{2(t''-t')}\Big)^{\frac{1}{p_k}}H(p_{k},\tau_{k},r_{k})\\
&\leqslant \big(4Ae^{2\sigma \widetilde T(1+\frac{2}{n})}\big)^{\sum\limits_{j=0}^k\frac{1}{p_{j+1}}}\Big(\frac{n\ell p_0}{2(t''-t')}+B+\frac{8\alpha^2n^2}{(r''-r')^2}+\frac{n}{2(t''-t')}\Big)^{\sum\limits_{j=0}^k\frac{1}{p_j}}\nu^{\sum\limits_{j=0}^k\frac{2(j+1)}{p_j}}H(p_0,\tau_0,r_0)\\
&\leqslant  \big(4Ae^{2\sigma \widetilde T(1+\frac{2}{n})}\big)^{\sum\limits_{j=0}^k\frac{1}{p_{j+1}}}\Big(B+\frac{n \ell p_0}{(t''-t')}+\frac{8\alpha^2n^2}{(r''-r')^2}\Big)^{\sum\limits_{j=0}^k\frac{1}{p_j}}\nu^{\sum\limits_{j=0}^k\frac{2(j+1)}{p_j}}H(p_0,\tau_0,r_0).
\end{split}
\end{equation*}
On the other hand, since $\tau_k\nearrow t''$ and $r_k\searrow r'$ as $k\nearrow +\infty$, we have
\begin{equation*}\label{0513001}
\begin{split}
&\ \ \ \ H(p_{k+1},\tau_{k+1},r_{k+1})\\
&=\Big(\int_{\tau_{k+1}}^{T'}\int_{B_{g(t)}(x,r_{k+1})}f^{p_{k+1}}dV_{g(t)}dt\Big)^{\frac{1}{p_{k+1}}}\\
&\geqslant \Big(\int_{t''}^{T'}\int_{B_{g(t)}(x,r')}f^{p_{k+1}}dV_{g(t)}dt\Big)^{\frac{1}{p_{k+1}}}\\
&=\Big(\int_{t''}^{T'}\int_{M}\Big(\chi  f\Big)^{p_{k+1}}dV_{g(t)}dt\Big)^{\frac{1}{p_{k+1}}}\\
&\geqslant \Big(e^{-B(T'-t'')}\int_{t''}^{T'}\int_{M}\Big(\chi  f\Big)^{p_{k+1}}dV_{g(T')}dt\Big)^{\frac{1}{p_{k+1}}}\\
&= \Big(e^{-B(T'-t'')}(T'-t'')\Vol_{g(T')}(Z)\Big)^{\frac{1}{p_{k+1}}}\Big(\frac{1}{(T'-t'')\Vol_{g(T')}(Z)}\int_{t''}^{T'}\int_{M}\Big(\chi  f\Big)^{p_{k+1}}dV_{g(T')}dt\Big)^{\frac{1}{p_{k+1}}},
\end{split}
\end{equation*}
where $Z=   \cup_{t\in (0,T)} B_{g(t)}(x,r) $ and we use Lemma \ref{202002003} in the last inequality,  and  the fact that $Z$ is compactly contained in $M$ and hence has bounded positive volume for all $t\in [0,T)$, and  the function $\chi:M \times (0,T) \to \{0,1\} \subseteq \R $ is defined by
$$\chi(y,t)=\left\{
\begin{aligned}
 &1,\ \ \ \ \ \ \ \ \ \ \  \text{ for } \ \ y \in B_{g(t)}(x,r'),\\
 &0,\ \ \ \ \ \ \ \ \ \ \ \text{ for } \ \ y \in   M\setminus B_{g(t)}(x,r').
\end{aligned}
\right.$$
Therefore, there holds
\begin{equation*}\label{0513003}
\begin{split}
&\ \ \Big(e^{-B(T'-t'')}(T'-t'')\Vol_{g(T')}(Z)\Big)^{\frac{1}{p_{k+1}}}\Big(\frac{1}{(T'-t'')\Vol_{g(T')}(Z)}\int_{t''}^{T'}\int_{M}\Big(\chi f\Big)^{p_{k+1}}dV_{g(T')}dt\Big)^{\frac{1}{p_{k+1}}}\\
&\leqslant \big(4Ae^{2\sigma \widetilde T(1+\frac{2}{n})}\big)^{\sum\limits_{j=0}^k\frac{1}{p_{j+1}}}\Big(B+\frac{n\ell p_0}{(t''-t')}+\frac{8\alpha^2n^2}{(r''-r')^2}\Big)^{\sum\limits_{j=0}^k\frac{1}{p_j}}\nu^{\sum\limits_{j=0}^k\frac{2(j+1)}{p_j}}H(p_0,\tau_0,r_0).
\end{split}
\end{equation*}
After letting $k$ go to $+\infty$, since $\sigma\widetilde T\leqslant2c_1\alpha^2$, we have
\begin{equation*}
\sup\limits_{(y,t) \in M\times[t'',T']}\Big(\chi(y,t)f(y,t)\Big)\leqslant C'_1\Big(B+\frac{n\ell p_0}{(t''-t')}+\frac{8\alpha^2n^2}{(r''-r')^2}\Big)^{\frac{n+2}{2p_0}}\Big(\int_{t'}^{T'}\int_{B_{g(t)}(x,r'')}f^{p_0}dV_{g(t)}dt\Big)^{\frac{1}{p_0}},
\end{equation*}
which is exactly
\begin{equation}\label{03220020}
\sup\limits_{\cup_{  t \in [t'',T']}     B_{g(t)}(x,r') \times \{t\}    }f  \leqslant C'_1\Big(B+\frac{n\ell p_0}{(t''-t')}+\frac{8\alpha^2n^2}{(r''-r')^2}\Big)^{\frac{n+2}{2p_0}}\Big(\int_{t'}^{T'}\int_{B_{g(t)}(x,r'')}f^{p_0}dV_{g(t)}dt\Big)^{\frac{1}{p_0}},
\end{equation}
where constant $C'_1=(4A)^{\frac{n}{2p_0}}(1+\frac{2}{n})^{\frac{(n+2)^2}{2p_0}}e^{\frac{2c_1\alpha^2 (n+2)}{p_0}}$.

Taking $r'=\frac{r}{2}$ and $r''=r$, we have
\begin{equation}\label{0430009}
\sup\limits_{\cup_{t\in[t'',T']} B_{g(t)}(x,\frac r 2) \times \{t\}}f\leqslant C'_1\Big(B+\frac{n\ell p_0}{(t''-t')}+\frac{32\alpha^2n^2}{r^2}\Big)^{\frac{n+2}{2p_0}}\Big(\int_{t'}^{T'}\int_{B_{g(t)}(x,r)}f^{p_0}dV_{g(t)}dt\Big)^{\frac{1}{p_0}}.
\end{equation}
Therefore, for any non-negative solution $f(t)$ of equation $(\ref{0322002})$, $p\geqslant2$ and $0<t'<t''<t<\widetilde T$,
\begin{equation}\label{03220022}
\sup\limits_{\cup_{s\in[t'',T']} B_{g(s)}(x,\frac r 2) \times \{s\}}f\leqslant C'_1\Big(B+\frac{n\ell p_0}{(t''-t')}+\frac{32\alpha^2n^2}{r^2}\Big)^{\frac{n+2}{2p_0}}\Big(\int_{t'}^{t}\int_{B_{g(s)}(x,r)}f^{p_0}dV_{g(s)}ds\Big)^{\frac{1}{p_0}}.
\end{equation}
Letting $t'=\frac{t}{2}$ and $t''\to t$, we have
\begin{equation}\label{03220023}
\sup\limits_{B_{g(t)}(x,\frac{r}{2})}f(t)\leqslant C'_1\Big(B+\frac{2n\ell p_0}{t}+\frac{32\alpha^2n^2}{r^2}\Big)^{\frac{n+2}{2p_0}}\Big(\int_{\frac{t}{2}}^{t}\int_{B_{g(s)}(x,r)}f^{p_0}dV_{g(s)}ds\Big)^{\frac{1}{p_0}}.
\end{equation}

Now we consider the case for $0<p<2$. For any $0<t'\leqslant s\leqslant T'<\widetilde T$, we define
\begin{equation}\label{0430001}
G(\rho,s)=\sup\limits_{\cup_{t\in [s,T']}  B_{g(t)}(x,\rho) \times \{t\} }f.
\end{equation}
We choose $p_0=2$ in $(\ref{03220020})$. Then for any $t'\leqslant s<t''< T'$ and $0<r'<\tilde r\leqslant r''\leqslant r$, there holds
\begin{equation*}
\begin{split}
G(r',t'')&\leqslant C'_1\Big(B+\frac{2n\ell}{(t''-s)}+\frac{8\alpha^2n^2}{(\tilde r-r')^2}\Big)^{\frac{n+2}{4}}\Big(\int_{s}^{T'}\int_{B_{g(t)}(x,\tilde r)}f^{2}dV_{g(t)}dt\Big)^{\frac{1}{2}}\\
&\leqslant C'_1\Big(B+\frac{2n\ell}{(t''-s)}+\frac{8\alpha^2n^2}{(\tilde r-r')^2}\Big)^{\frac{n+2}{4}}\Big(\big(\sup\limits_{\cup_{t\in [s,T']} B_{g(t)}(x,\tilde r) \times \{t\}}f\big)^{2-p}\int_{s}^{T'}\int_{B_{g(t)}(x,\tilde r)}f^{p}dV_{g(t)}dt\Big)^{\frac{1}{2}}\\
&=C'_1\Big(B+\frac{2n\ell}{(t''-s)}+\frac{8\alpha^2n^2}{(\tilde r-r')^2}\Big)^{\frac{n+2}{4}}G(\tilde r,s)^{\frac{2-p}{2}}\Big(\int_{s}^{T'}\int_{B_{g(t)}(x,\tilde r)}f^{p}dV_{g(t)}dt\Big)^{\frac{1}{2}}\\
&\leqslant\frac{2-p}{2}\Big(C_3'G(\tilde r,s)^{\frac{2-p}{2}}\Big)^{\frac{2}{2-p}}+\frac{p(C^{'-1}_3C'_1)^{\frac{2}{p}}}{2}\Big(B+\frac{2n\ell}{(t''-s)}+\frac{8\alpha^2n^2}{(\tilde r-r')^2}\Big)^{\frac{n+2}{2p}}\Big(\int_{s}^{T'}\int_{B_{g(t)}(x,\tilde r)}f^{p}dV_{g(t)}dt\Big)^{\frac{1}{p}}\\
&\leqslant\frac{2-p}{2}{C'_3}^{\frac{2}{2-p}}G(\tilde r,s)+\frac{p({C'_3}^{-1}C'_1)^{\frac{2}{p}}}{2}\Big(B+\frac{2n\ell}{(t''-s)}+\frac{8\alpha^2n^2}{(\tilde r-r')^2}\Big)^{\frac{n+2}{2p}}\Big(\int_{t'}^{T'}\int_{B_{g(t)}(x,r'')}f^{p}dV_{g(t)}dt\Big)^{\frac{1}{p}}.
\end{split}
\end{equation*}
We choose $C'_3$ depending only on $p$ such that $(2-p){C'_3}^{\frac{2}{2-p}}\leqslant1$. Denote $\beta=\frac{n+2}{2p}$ and 
\begin{equation}\label{}
Q=\frac{p(C^{'-1}_3C'_1)^{\frac{2}{p}}}{2}\Big(\int_{t'}^{T'}\int_{B_{g(t)}(x,r'')}f^{p}dV_{g(t)}dt\Big)^{\frac{1}{p}},
\end{equation}
then we have
\begin{equation}\label{0430002}
G(r',t'')\leqslant\frac{1}{2}G(\tilde r,s)+Q\Big(B+\frac{2n\ell}{(t''-s)}+\frac{8\alpha^2n^2}{(\tilde r-r')^2}\Big)^{\beta}.
\end{equation}
We choose $0<\gamma<1$ depending on $n$ and $p$ such that $2\gamma^{2\beta}>1$, and we let
\begin{equation}\label{0430003}
\begin{split}
&\ t_0=t'',\ \ \ \ t_{i+1}=t_i-(1-\gamma)\gamma^{i}(t''-s),\ \ \ \ \ \ then\ t_i\searrow s\ as\ i\to+\infty,\\
&\ r_0=r',\ \ \ r_{i+1}=r_i+(1-\gamma)\gamma^{i}(\tilde r-r'),\ \ \ then\ r_i\nearrow \tilde r\ as\ i\to+\infty.
\end{split}
\end{equation}
Then for all $i\in\mathbb{N}$,
\begin{equation}\label{0430004}
\frac{1}{t_i-t_{i+1}}=\frac{\gamma^{-i}}{(1-\gamma)(t''-s)}\ \ \ and\ \ \ \frac{1}{r_{i+1}-r_i}=\frac{\gamma^{-i}}{(1-\gamma)(\tilde r-r')}.
\end{equation}
By iteration, since $0<\gamma<1$, we have
\begin{equation}\label{0430005}
\begin{split}
&\ \ \ G(r',t'')=G(r_0,t_0)\\
&\leqslant\frac{1}{2}G(r_1,t_1)+Q\Big(B+\frac{2n\ell}{(t_0-t_1)}+\frac{8\alpha^2n^2}{(r_1-r_0)^2}\Big)^{\beta}\\
&\leqslant\frac{1}{2^2}G(r_2,t_2)+\frac{Q}{2}\Big(B+\frac{2n\ell}{(t_1-t_2)}+\frac{8\alpha^2n^2}{(r_2-r_1)^2}\Big)^{\beta}+Q\Big(B+\frac{2n\ell }{(t_0-t_1)}+\frac{8\alpha^2n^2}{(r_1-r_0)^2}\Big)^{\beta}\\
&=\frac{1}{2^2}G(r_2,t_2)+\frac{Q}{2}\Big(B+\frac{2n\ell \gamma^{-1}}{(1-\gamma)(t''-s)}+\frac{8\alpha^2n^2\gamma^{-2}}{(1-\gamma)^2(\tilde r-r')^2}\Big)^{\beta}\\
&\ \ \ +Q\Big(B+\frac{2n\ell}{(1-\gamma)(t''-s)}+\frac{8\alpha^2n^2}{(1-\gamma)^2(\tilde r-r')^2}\Big)^{\beta}\\
&\leqslant\frac{1}{2^{k+1}}G(r_{k+1},t_{k+1})+Q\sum\limits_{i=0}^k\Big(B\gamma^{2i}+\frac{2n\ell\gamma^i}{(1-\gamma)(t''-s)}+\frac{8\alpha^2n^2}{(1-\gamma)^2(\tilde r-r')^2}\Big)^{\beta} \Big(\frac{1}{2\gamma^{2\beta}}\Big)^i\\
&\leqslant\frac{1}{2^{k+1}}G(r_{k+1},t_{k+1})+Q\Big(B+\frac{2n\ell}{(1-\gamma)(t''-s)}+\frac{8\alpha^2n^2}{(1-\gamma)^2(\tilde r-r')^2}\Big)^{\beta}\sum\limits_{i=0}^k \Big(\frac{1}{2\gamma^{2\beta}}\Big)^i
\end{split}
\end{equation}
Since $G(r_k,t_k)\to G(\tilde r,s)$ as $k\to+\infty$, after letting $k\to+\infty$, we have
\begin{equation}\label{0430006}
G(r',t'')\leqslant \frac{2\gamma^{2\beta}Q}{2\gamma^{2\beta}-1}\Big(B+\frac{2n\ell}{(1-\gamma)(t''-s)}+\frac{8\alpha^2n^2}{(1-\gamma)^2(\tilde r-r')^2}\Big)^{\beta}.
\end{equation}
Taking $s=t'$ and $\tilde r=r''$, we have
\begin{equation}\label{0430007}
\begin{split}
& \sup\limits_{\cup_{t \in [t'',T']} B_{g(t)}(x,r')\times \{t\} }f\\
&\leqslant \frac{\gamma^{2\beta}p(C^{'-1}_3)^{\frac{2}{p}}(4A)^{\frac{n}{2p}}(1+\frac{2}{n})^{\frac{(n+2)^2}{2p}}e^{\frac{2c_1\alpha^2(n+2)}{p}}}{2\gamma^{2\beta}-1}\times\\
&\ \ \ \Big(B+\frac{2n\ell}{(1-\gamma)(t''-t')}+\frac{8\alpha^2n^2}{(1-\gamma)^2(r''-r')^2}\Big)^{\beta}\Big(\int_{t'}^{T'}\int_{B_{g(t)}(x,r'')}f^{p}dV_{g(t)}dt\Big)^{\frac{1}{p}}
\end{split}
\end{equation}
Therefore, for any $0<p<2$, there exist positive constants $C'_2$ and $C'_4$ depending only on $p$, $A$, $c_0$ and $n$ such that
\begin{equation}\label{0430008}
\begin{split}
& \sup\limits_{\cup_{ t\in [t'',T']} B_{g(t)}(x,r') \times \{t\}}f \\
& \leqslant C'_2\Big(B+\frac{2n\ell}{(1-\gamma)(t''-t')}+\frac{8\alpha^2n^2}{(1-\gamma)^2(r''-r')^2}\Big)^{\frac{n+2}{2p}}\Big(\int_{t'}^{T'}\int_{B_{g(t)}(x,r'')}f^{p}dV_{g(t)}dt\Big)^{\frac{1}{p}}\\
&\leqslant C'_2\frac{1}{(1-\gamma)^{\frac{n+2}{p}}}\Big(B(1-\gamma)^2+\frac{2n\ell(1-\gamma)}{(t''-t')}+\frac{8\alpha^2n^2}{(r''-r')^2}\Big)^{\frac{n+2}{2p}}\Big(\int_{t'}^{T'}\int_{B_{g(t)}(x,r'')}f^{p}dV_{g(t)}dt\Big)^{\frac{1}{p}}\\
&\leqslant C'_4\Big(B+\frac{2n\ell}{(t''-t')}+\frac{8\alpha^2n^2}{(r''-r')^2}\Big)^{\frac{n+2}{2p}}\Big(\int_{t'}^{T'}\int_{B_{g(t)}(x,r'')}f^{p}dV_{g(t)}dt\Big)^{\frac{1}{p}}.
\end{split}
\end{equation}
Taking $r'=\frac{r}{2}$ and $r''=r$, we have
\begin{equation}\label{0430009new}
\sup\limits_{\cup_{t \in [t'',T']} B_{g(t)}(x,\frac{r}{2}) \times \{t\}}f\leqslant C'_4\Big(B+\frac{2n\ell}{(t''-t')}+\frac{32\alpha^2n^2}{r^2}\Big)^{\frac{n+2}{2p}}\Big(\int_{t'}^{T'}\int_{B_{g(t)}(x,r)}f^{p}dV_{g(t)}dt\Big)^{\frac{1}{p}}.
\end{equation}
Therefore, for any non-negative solution $f$ of equation $(\ref{0322002})$, $0<p<2$ and $0<t'<t''<t< \widetilde T\leqslant T$,
\begin{equation}\label{03220022new}
\sup\limits_{\cup_{s \in [t'',t]} B_{g(s)}(x,\frac{r}{2}) \times \{s\}}f\leqslant C'_4\Big(B+\frac{2n\ell}{(t''-t')}+\frac{32\alpha^2n^2}{r^2}\Big)^{\frac{n+2}{2p}}\Big(\int_{t'}^{t}\int_{B_{g(s)}(x,r)}f^{p}dV_{g(s)}ds\Big)^{\frac{1}{p}}.
\end{equation}
Letting $t'=\frac{t}{2}$ and $t''\to t$, we have
\begin{equation}\label{03220023new}
\sup\limits_{B_{g(t)}(x,\frac{r}{2})}f(t)\leqslant C'_4\Big(B+\frac{4n\ell}{t}+\frac{32\alpha^2n^2}{r^2}\Big)^{\frac{n+2}{2p}}\Big(\int_{\frac{t}{2}}^{t}\int_{B_{g(s)}(x,r)}f^{p}dV_{g(s)}ds\Big)^{\frac{1}{p}},
\end{equation}
which finishes  the proof.
\end{proof} 
 
\section{An Application}\label{AnApplication}
As an application of our results from the previous sections, we prove Theorem \ref{OrbifoldThmIntro}.

\noindent {\it Proof of Theorem \ref{OrbifoldThmIntro}.}
From (2) of  Theorem \ref{mainthmIntro}, we have the scale invariant inequality
  \begin{eqnarray}
  \int_N |\Rm_{g(t)}|_{g(t)}^2 dV_{g(t)}\leq K_0
 \end{eqnarray}
 for all $t\in [0,  T),$ for some fixed $K_0< \infty$. 
From  (3) of Theorem \ref{mainthmIntro}   we have 
 \begin{eqnarray}
 \int_S^T \int_{B_{g(t)}(x,K\sqrt{T-t})}   |\Rc_{g(t)}|_{g(t)}^{2+ \alpha^3} dV_{g(t)} dt \leq c(K,\ldots) (T-S)^{1+ \frac{\al}{16}}  
 \end{eqnarray}
for all $0<S<T,$ $x \in N,$ $K \in [1,\infty)$ with  $K\sqrt{T-S } \leq 1.$  
  
 Hence, for any $x \in N,$ we can find $t_i \to T$ such that 
  \begin{eqnarray}
 	\int_{B_{g(t_i)}(x, K_i \sqrt{T-t_i})}   |\Rc_{g(t)}|_{g(t)}^{2+\alpha^3} dV_{g(t)}  
 	\leq \ep_i  
 \end{eqnarray}
 where $K_i \to \infty$ and $\ep_i \to 0$ as $i \to\infty.$
 We call such  sequences of times $(t_i)_{i \in \N}$, { \it $x-$good time sequences}. 
 We first consider {\it regular points} $x \in N$. That is $x \in N$   such that  a
    $x-$good time sequence $(t_i)_{i \in \N}$ 
  exists with 
 \begin{eqnarray}
 	\int_{B_{g(t_i)}(x, 4R \sqrt{T-t_i})}   |\Rm_{g(t)}|_{g(t)}^{2} dV_{g(t)}  
 	\leq \ep_0
 \end{eqnarray}
 for some fixed small $\ep_0>0$ and some fixed large $R>0.$
 After scaling the solution by $\frac{1}{(T-t_i)}$ and translating in time,  
 that is setting $g_i(t) = \frac{1}{(T-t_i)} g(\cdot, (T-t_i)(  t+1)  + t_i ),$    we have
 \begin{eqnarray}
  \int_{B_{g_i(-1)}(x,4R )}  |\Rc_{g_i(-1)}|_{g_i(-1)}^{2+ \alpha^3}     dV_{g_i(-1)}  \leq \ep_i    \label{Riccismallsc} \\
   \int_{B_{g_i(-1)}(x,4R )}  |\Rm_{g_i(-1)}|_{g_i(-1)}^{2}     dV_{g_i(-1)}  \leq \ep_0   \label{Riemsmallsc}
 \end{eqnarray}
 In the following, we denote all constants $c(\ep_0,\ldots)$  with the property that $c(\ep_0,\ldots) \to 0$ 
 as $\ep_0 \to 0$  simply by $\ep,$ and we refer to the metric $g_i(t)$ as $g(t)$ for ease of reading.  
 As explained in  the paper \cite{Anderson} (Proof of Main Lemma 2.2, respectively Remark 2.3 (ii) therein : see 
 the proof of Theorem B.7 in the arxiv version of the paper \cite{MSIM2} for some more details in the case $p=12$, or Appendix A in this paper), \eqref{Riccismallsc},  \eqref{Riemsmallsc} and the non-inflating/non-collapsing estimates,    $\ti 1$) of the Remark following Theorem \ref{nonin2Intro}, 
 guarantee that 
  there are  $W^{1, p}$ 
 harmonic coordinates , with $p = 4 +2\al>n=4$,   on $B_{g(-1)}(x,V)$: that is coordinates 
 $\psi  :  V \subseteq B_{g(-1)}(x, 2R) \to \B_{R}(0) = \psi(V)   $ 
with   the following properties (here $g_{ij}$ is the metric in
these coordinates)

\begin{itemize}

\item[(i)] $ (1 - \ep) \de_{ij} < g_{ij} < (1+\ep )\de_{ij}$ on $\B_{R}(0)$
\item[(ii)]  $ R^{2-\frac{n}{p} } \| D g \|_{L^{p}(\B_{R}(0))} <\ep$
\item[(iii)] $\psi: V \to \psi(V) = \B_{R}(0),$  $\psi(x) = 0,$ and $\psi$  is harmonic  : $\lap_g \psi^k = 0
  $ on $V$ for
  all $k \in \{1,\ldots,n\}.$  
\end{itemize}
 
In particular, 
 we can find coordinates  $\phi: B_{g(-1)}(x,  R) \to U \subseteq \R^n$ such that $(1-\ep) \de_{ij} \leq g_{ij} \leq (1+\ep) \de_{ij}$ on $U.$ The choice of large $  R$ then implies that 
 $|\Rm_{g(t)}|_{g(t)}    \leq \frac{\ep}{ 1+ t}$  
 on $   B_{g(t)}(y,  \frac R 4)$  for all   $t \in ( -1,0), $ $y \in  B_{g_i(-1)}(x, \frac R 2)$ in view of Perelman's Pseudolocality result, see \cite{GP}. Translating the time interval   $ ( -1,0)$ to $ (0,1)$ that is setting $\hat g( \cdot, \hat  t) =  g(\cdot, \hat t -1)$ and calling the solution $g$ again, we see that
   $|\Rm_{g(t)}|_{g(t)}  \leq \frac{\ep}{t}$   on 
 $  B_{g(t)}(y,  \frac R 4) $ for all $t \in ( 0,1) $ for any $y \in  B_{g(0)}(x,  \frac R 2 ) .$ 
 
     We note that $f(\cdot,t):= \sqrt{|\Rc_{g(t)} |_{g(t)}^2(\cdot) +\ep}$ is a smooth  function   of space and time  which satisfies
 $\partt f \leq \lap_g f +  c(n)f |\Rm_g|_g $  where $|\Rm_{g(t)}|_{g(t)}\leq \frac{\ep}{t}$ on   $B_{g(t)}(y,  R/2)$  for all $t \in ( 0,1),$ as we now explain   (in the following, $\lap$ is the Laplacian with respect to $g$  and
 $\nabla $ is the covariant derivative with respect to $g$): 
  
 \begin{eqnarray}
&&\partt \sqrt{\ep + |\Rc_g |_g^2}   \cr
&&= \frac{1}{2\sqrt{\ep + |\Rc_g |_g^2}} ( \lap  |\Rc_g |_g^2 -2|\grad \Rc_g|_g^2 + 4 \Rm_g(\Rc_g,\Rc_g) )\\
&&\text {and } \cr
&& \lap \sqrt{\ep + |\Rc_g |_g^2 } \cr
&& = g^{ij}\nabla_j(\frac{1}{2\sqrt{\ep + |\Rc_g |_g^2}} \nabla_i|\Rc_g |_g^2 ) \cr
&& =    \frac{1}{2\sqrt{\ep + |\Rc_g |_g^2}} \lap |\Rc|^2       -\ \ \ \ \frac{1}{  (\ep + |\Rc_g |_g^2)^{\frac 3 2} } |\Rc_g |_g^2 g^{ij}\nabla_i|\Rc_g |_g  \nabla_j|\Rc_g |_g \cr
&& \geq     \frac{1}{2\sqrt{\ep + |\Rc|^2}}   \lap |\Rc|^2  
 -\frac{1}{  \sqrt{\ep + |\Rc_g |_g^2} }   |\grad \Rc_g|_g^2  \cr
 && =  \frac{1}{2\sqrt{\ep + |\Rc|^2}} ( \lap |\Rc|^2   -2  |\grad \Rc|^2) 
\end{eqnarray} 
  and hence 
 \begin{eqnarray} 
 \partt \sqrt{\ep + |\Rc_g |_g^2} \leq   \lap_g \sqrt{\ep + |\Rc_g |_g^2} +    c(n)|\Rm_g|_g \sqrt{\ep + |\Rc_g |_g^2}
 \end{eqnarray}
 Furthermore,  (ii) of Theorem  \ref{LiuSim1}  gives us 
\begin{eqnarray}
&&  \int_{\frac t 2}^t \int_{B_{g(s)}(y,\frac{R}{4})}  |\Rc_{g(s)}|_{g(s)}^4   dV_{g(s)} ds  \cr
 && \leq  \frac{\ep}{t^{1-\alpha}}  + \ep  t^{1+\alpha}  \sup \{ |\Rc_{g(s)}|_{g(s)}^2(y,s)   \ | \  s \in [\frac t 2 , t] , \ y \in  B_{g(s)}(y,\frac{R}{2}) \}
\cr
&& \leq \frac{3\ep}{t^{1-\alpha}},\label{betterint}
 \end{eqnarray}
for all $t\in (0,1),$
where    the constant $\hat c$  of  (ii) of   Theorem  \ref{LiuSim1}   has been replaced by the better constant $\ep$: we have scaled the original solution by a constant larger than one which is large as we like,  and this leads to $\ep$ in the inequality above being as small as we like, as explained in Remark \ref{remarkext}. 
   Using      Theorem \ref{MoserIteration} for $f $  and  $p=4$, and \eqref{betterint},   
we see for all $y \in B_{g(0)}(x,\frac R 2)$,   that 
\begin{eqnarray}
 \sup_{B_{g(t)}(y,\frac{R}{8}) } |\Rc_{g(t)}|_{g(t)}  && \leq C(R,\si_0,\si_1, \ldots) \frac{1}{t^{\frac 3 4}} \Big( \int_{\frac t 2}^t \int_{B_{g(s)}(y,\frac{R}{4})}(\ep+|\Rc_{g(s)}|_{g(s)}^4)  dV_{g(s)} ds  \Big)^{\frac 1 4}\cr
&& \leq \frac{C }{ t^{\frac 3 4}}  \Big(    \frac{\ep}{t^{1-\alpha}}   \Big)^{\frac 1 4} \cr
&& \leq \frac{\ep }{ t^{1- \frac{\alpha}{4} }} 
\end{eqnarray}
which is integrable in time.

After integrating the equation $\partt g (y,t) = -2\Rc(g)(y,t)$ with respect to time for each $ y  \in B_{g(0)}(x,\frac{R}{5}),$ we also get
  $(1-2\ep) \de_{ij} \leq g_{ij} \leq (1+2\ep) \de_{ij}$ on $B_{g(0)}(x,\frac{R}{5})$,   for all $t\in (0,  1). $  Translating the time interval $(0,1)$ back to $(-1,0)$ again, we see that 
    $g_i(t) \to (g_i)(0)$ smoothly for some smooth Riemannian metric $g_i(0)$ on the open set $V:= B_{g(-1)}(x,\frac {R}{10})$ as $t\upto 0$. 
  Scaling back to the original solution, we see that, that for any     $t_i \in (0,T)$ close enough to $T,$ with $(t_j)_{j\in N}$ an $x-$ good time sequence,   that $g(t) \to g(T)$  on $V= B_{g(t_i)}(x,\sqrt{T-t_i})$ smoothly as $t\upto T$ ($i$ is now fixed) where $g(T)$ is a well defined smooth Riemannian metric on $V$ satisfying $ g(T) (1-\ep) \leq g(t) \leq g(T)(1+\ep)$ for $  t \in (t_i,T).$  
Compare  Theorem 4.5 and the proof thereof in  \cite{MSIM2}. \\
Now we consider the singular points, that is  points $x \in N$ for which for any $x$-good time sequence $(t_i)_{i\in \N},$ we have 
$$\int_{B_{g(t_i)}(x,4R\sqrt{T-t_i})}  |\Rm_{g(t_i)}|_{g(t_i)}^2dV_{g(t_i)} > \ep_0,$$
for $i$ large enough, where the $t_i$ also satisfy $t_i \upto T$ as $i\to \infty$.    
By clustering singular points together for which $\int_{B_{g(t_i)}(x,2R\sqrt{T-t_i})}  |\Rm_{g(t_i)}|_{g(t_i)}^2dV_{g(t_i)} > \ep_0,$
we see, using $\int_N |\Rm_{g(t)}|^2dV_{g(t)} \leq K_0$ and the estimates proved for the regular points above,  
that there   exist  points $p_1(t_i), \ldots , p_L(t_i) $, $L=L(\si_0,\si_1,R,K_0,\ep_0)$ and some constant $\La>0 $, such that
$\int_{B_{g(t_i)}(p_j,\La \sqrt{T-t_i})}  |\Rm_{g(t_i)}|_{g(t_i)}^2dV_{g(t_i)}   > \ep_0,$ and that 
$d(g(t_i))(p_k,p_j) \geq N(i) \sqrt{T-t_i}$ for all $k \neq j,$ where $N(i) \to \infty$ as $i\to \infty,$ and that 
all points $x \in N \backslash ( \cup_{j=1}^L  B_{g(t_i)}(p_j,\La \sqrt{T-t_i}) )$  are regular points: See the clustering argument in \cite{MSIM2}.

We consider one of these points $p_j$ and denote it by $q,$ and we scale  and translate the solution  as above: $g_i( t) = \frac{1}{(T-t_i)} g(\cdot, (T-t_i)(  t +1)  + t_i ).$ 
As explained in   Appendix B of this  paper, after increasing $\Lambda$ if necessary, 
 $ B_{g_i(-1)}(q,\La) \backslash B_{g_i(-1)}(q,1)$ has only one component  for $i$ large enough. 
Using the argument in the proof of Theorem 5.1 in  \cite{MSIM2},   we see that the distance with respect to  $d(g_i(t))$ from $q$ to 
$\boundary B_{g_i(-1)}(q,100\La) $ is bounded by some constant $A=C(n,\si_0,\si_1)\La^5$ for $t\in (-1,0].$ Also, $ B_{g_i(-1)}(q,\hat C(\La,\si_0,\si_1,p=4+\al) ) \backslash  B_{g_i(-1)}(q,\La) $ is connected  for $i$ large enough, for some $\hat C(\La,\si_0,\si_1,p=4+\al)$ as we now explain . 
Scale  $g_i(-1)$ by  $\frac 1 {\La^2},$ that is $h_i(\cdot):= \frac 1 {\La^2}g_i(-1)(\cdot)$. We still have $\Rc(h_i)_{-}\to 0$ as $i \to \infty$ locally in $L^p$ and, hence Appendix B is applicable, and hence
$ B_{h_i}(q,R) \backslash B_{h_i}(q,1)$  is connected for some $R= R(\si_0,\si_1,n,p)>0$.  
 Scaling back gives us  $ B_{g_i(-1)}(q,\hat C(\La,\si_0,\si_1) ) \backslash  B_{g_i(-1)}(q,\La) $ is connected.  Points in 
$ B_{g_i(-1)}(q,\hat C) \backslash  B_{g_i(-1)}(q,\La) $ are regular points and  can be joined to one another by curves in  $ B_{g_i(-1)}(q,\hat C) \backslash  B_{g_i(-1)}(q,\La) $ of length, with respect to $g_i(t)$ for any  $t\in (-1,0)$, less then $C(\La,\si_0,\si_1),$ as explained in the proof of Theorem 5.1 in  \cite{MSIM2}. 
This means $d(g_i(t))(q, y) \leq C_3(\La,\si_0,\si_1)$ for all $t \in (-1,0)$ for all   $y \in \boundary B_{g(-1)}(q,100\La).$\\
 Performing this analysis  for each point $p_j$ and  scaling back to the original solution, we see that this means  that ${\rm sing}(N,g(t)) \subseteq \cup_{j=1}^L  B_{g(t_i)}(p_j(t_i), C_3 \sqrt{T-t_i}) .$

Equipped with the above distance estimates  for the regular and singular part,   
as  explained in Chapter 6 of \cite{MSIM2},  we    now  have: \\
(a) $(N,d(g(t))) \to (X,d_X),$ uniformly as $t \upto T,$ 
where $X: = \{ [x] \ | \ x \in N\}$ and $[x]=[y]$ if and only if  $d(x,y,t) \to 0$ as $t\upto T\},$ 
and $d_X([x],[y]):= \lim_{t\upto T} d(g(t)(x,y)$ is well defined, and $(X,d_X)$ is a  metric space,
\\ (b) Defining $f:   N \to X$ by  $f(x) := [x]$, we see that there are at most finitely many points $x_1, \ldots, x_L \in X$ such that
$ f: V:= f^{-1}(X\backslash \{ x_1, \ldots, x_L \}) \to Z:= X\backslash \{ x_1, \ldots, x_L \}$ is a   homeomorphism and there is a natural smooth structure on $Z$ defined by $f$, and a natural smooth Riemannian metric $l$ on $Z$ given by
$l:= \lim_{t\upto T} f_*(g(t))$  in the smooth sense, and the metric space induced by $l$ on $X$ agrees with $(X,d_X)$ locally: In the case that $M=N,$ this holds  globally. \\
(c) The results of Chapter 7 of \cite{MSIM2} also hold, after  changing  (7.12) in the approximation Theorem, Theorem 7.1  to 
$\int_{B_{g_i}(p_j,N_i,)} |\Rc(g_i)|^{2+ \al^3} dg_i \to 0$  as $i\to \infty,$ where $N_i \to \infty$ as $i\to \infty$.
In particular, for any fixed $j\in \{1, \ldots, L\},$   we have $(X, i \cdot  d_X,x_j) \to \R^n/ \Gamma_j$ as $i\to \infty$ where $\Gamma_j \subseteq O(4)$ is a finite subgroup, the convergence being smooth on compact subsets away from the tip of the cone and otherwise in the distance (Gromov-Hausdorff) sense.

The results of Chapter 8 of \cite{MSIM2}  can now be applied to show that $(X,d_X)$ is a $C^0$ Riemannian orbifold,  which  is smooth away from the non-manifold points, 
in the sense of Theorem 8.3 in \cite{MSIM2}. 
Much of the proof of Theorem 8.3 is not necessary here, as a fair amount of the proof there is   concerned with the connectedness of $X \backslash \{ x_1, \ldots x_L\}$: this has been shown here at an earlier stage.

Chapter 9 of \cite{MSIM2}   now guarantees that one can extend the flow past time $T$ to $T+\si$ for some $\si>0$ using the Orbifold Ricci flow, if $M=N$ is a closed  smooth manifold.
 $\Box$

\section{ Appendix A }

We define the $W^{1, p}$  harmonic radius as follows.
\begin{defi}
Let $(M,g)$ be a smooth Riemannian manifold and $p \in M, \al \in (0,1).$ 
For  $p = 4 +\al>4$,   we define $r_{h,p}(y)$ to be the supremum over all $r>0$ such that there exist smooth  coordinates 
 $\psi  :  V \subseteq B_{g}(x, 2r) \to \B_{r}(0) = \psi(V)   $  where   $  B_{g}(x, 2r)$ is  compactly contained in $M$ 
and such that  the following holds  (here $g_{ij}$ is the metric in
these coordinates)

\begin{itemize}

\item[(i)] $ (1-\ep) \de_{ij} < g_{ij} < (1+ \ep)  \de_{ij}$ on $\B_{r}(0)$
\item[(ii)]  $ r^{2-\frac{n}{p} } \| D g \|{L^{p}(\B_{r}(0))} <\ep$
\item[(iii)] $\psi: V \to \psi(V) = \B_{r}(0),$  $\psi(x) = 0,$ and $\psi$  is harmonic  : $\lap_g \psi^k = 0
  $ on $V$ for
  all $k \in \{1,\ldots,n\}.$  
\end{itemize}
\end{defi} 

\begin{thm}\label{hrtheo}
For any  constants  $0< c_0,\al , \si_0, \si_1, \ldots $ there exists an $\ep_0>0$ (small) such  that the following holds.

 $(M^4,g)$ be a smooth manifold without boundary (not necessarily complete)  and
${B}_{3}(q,d(g)) \subseteq \subseteq M$ be an arbitrary ball which is compactly
contained in $M$. 
Assume that
\begin{itemize}
\item [(a)] $\int_{ B_{3}(q)  } |\Rm|^2 d\mu_{g} \leq \ep_0$
and $\int_{B_3(q)} |\Ric|^{2+\al} d\mu_{g} \leq c_0$, 
\item[(b)]  $\si_0 r^4 \leq \vol(B_r(x)) \leq \si_1 r^4$ for all $r
  \leq 1 $, for all $x \in B_3(q)$,
\end{itemize}
where $\ep_0= \ep_0(\si_0,\si_1,\ep)>0$ is small enough.
Then there exists a constant $V= V(\si_0,\si_1)>0$ 
such that 
\begin{eqnarray}
r_{h,4+2\al}(g)(y) \geq V \dist_{g}(y, \boundary(B_{1}(q)))
 \end{eqnarray}
for all $y \in B_1(q)$.
Here $B_3(q)$ is the {\it reference ball} used in the definition of the
harmonic radius.
\end{thm}

\begin{proof}

The proof is by contradiction, and essentially  the same as that presented in the proof of Theorem  B.7 in the Arxiv version of the paper  \cite{MSIM2}. 
The  proof method given there, is essentially that given in the proof of Main
Lemma 2.2 of  \cite{Anderson} (see Remark 2.3 (ii) there) using some notions from
\cite{AnCh} on the $W^{1,p}$ harmonic radius.

Assume the result is false. Then,   as  explained in \cite{MSIM2}, we obtain a sequence of points $y_i$ and Riemannian metrics (after scaling), $g_i$ defined on $M_i$ such that $B_{g(i)}(y_i,i) $  is compactly contained in $M_i$ and
\begin{eqnarray}
&& r_{h}(g(i))(y_i) =1,\label{firsthar}\\ 
&& r_{h}(g(i))(y) \geq \frac 1 2 \  \text{ for all }\  y \in B_{g(i)}(y_i,i)\label{secondhar}\\
&& \int_{  B_{g(i)}(y_i,i)} |  \Rc|^{2+\al} d\vol_{  g(i)}    \to 0 \mbox{  as  } \ i \to
\infty  \label{Rceqhar} \\
&& \int_{  B_{g(i)}(y_i,i) } | \Riem|^2 du_{  g(i)}  \leq \frac{1}{i} \to 0 \mbox{  as  } \ i \to
\infty  \label{Rmeqhar} 
\end{eqnarray}
We find a covering of $   B_{g(i)}(y_i,L)$ for any fixed large $L$ by  harmonic coordinates,  with  charts  
$\psi_{i,s}:  V_{i,s} \to B_{\frac 1 2}(0).$  
Exactly as explained in \cite{MSIM2}, it then follows that the Gromov-Hausdorff limit 
$\lim_{GH}(M_i,g(i),y_i) = (X,h,p) $ exists and is a $C^0$ Riemannian manifold.
Writing $h$ in the limiting charts and $g(i)$ in the harmonic coordinate charts we see that 
$g(i) \to h$ {\bf weakly } locally (that is written in these coordinates)  in $W^{1,4+2\alpha}$.
 
Using the facts, that the metric itself solves the equation
$$g^{ab}\partial_a \partial_b g_{kl} = (g^{-1}*g^{-1}*\partial g* \partial g)_{kl}- 2 \Rc(g)_{kl}$$ in harmonic coordinates,  and that equation \eqref{Rceqhar} holds, 
we may use the elliptic $L^p$ theory and the Sobolev embedding theorem,  to obtain  that this convergence is in fact {\bf strong} locally, and occurs also in the $W^{2,2+\alpha}$ sense strongly locally.
It then follows that $h$ satisfies $$h^{ab}\partial_a \partial_b h_{kl} = (h^{-1}*h^{-1}*\partial h* \partial h)_{kl} $$ in the $W^{2,2+\alpha}$ sense and we may use the 
  elliptic $L^p$ theory and the Sobolev embedding theorem again, to show that   the limit metric $h$ is $C^{\infty} $
  in the constructed coordinates. 
The transition maps 
$s_i:= s_{i,r,t}:= (\phi_{i,r})^{-1} \of \phi_{i,t}$ also satisfy an elliptic equation, 
\begin{eqnarray}
0 &&= g(i)^{ab}\partial_a \partial_b s_i 
\end{eqnarray}
Using elliptic $L^p$ theory and the Sobolev embedding theorem now for the transition functions, with the fact that $g(i)$ converges locally in the $C^{\beta}$ norm, we obtain that the transition functions of the limiting manifold  are $C^{\infty}$.
The non-collapsing condition, and equation \eqref{Rmeqhar} now show us that $(X,h,p)$ is in fact smooth and isometric to the standard four dimensional euclidean space with the standard metric.

For the readers convenience, we present more details.

Let us denote  the metric $g(i)$ in
the local coordinates given by $\phi_{i,r}$ by
$g(i)_{kl}$, and $h$ in the local harmonic coordinates given by $\phi_r$ by $h_{kl}$: $r$ is fixed for the moment.

By construction  we have: $g(i)_{kl} \to h_{kl}$ in $C^{0,\al}(B_\ep(u))$ for any 
$B_{2\ep}(u) \subseteq B_{1/100}(0) \subseteq \R^4$  and 
  $g(i)_{kl}
\to h_{kl}$ {\bf weakly} in $W^{1,4+2\al}(B_{\ep}(u))$, after taking a
subsequence (see for example   Section 8  of \cite{Evans}).
In particular 
this tells us that $g(i) \to h$
{\it strongly} in $L^{p}(K)$ for any $p <\infty.$
The metric $g_{kl} = g(i)_{kl}$ satisfies
\begin{eqnarray}
 g^{ab}\partial_a\partial_bg_{kl} =  (g^{-1} * g^{-1} * \partial g
 * \partial g)_{kl} -2 \Rc(g)_{kl} \label{toodle}
\end{eqnarray}
smoothly in $B_{1/100}(0)$, since the coordinates are  harmonic.
Here the $(0,2)$ tensor $(g^{-1} * g^{-1} * \partial g * \partial g)$ can be written
explicitly, but in order to make the argument more readable we use
this star notation. 
This tensor has the property that is linear in all of its terms: 
for $Z,\ti Z,W,\ti W$ symmetric and positive definite local $(2,0)$ Tensors,
and $R,\ti R, S,\ti S$ local $(0,3)$ Tensors,
we have $Z*W*R*S$ is a local $(0,2)$ Tensor, with the property that
$ (Z + \ti Z)*W*R*S = Z*W*R*S + \ti Z*W*R*S$ and
$Z*(W+ \ti W)*R*S = Z*W*R*S + Z*\ti W*R*S$ and so on.
Furthermore $|Z*W*R*S|_g \leq c|Z|_g |W|_g |\Sc|_g |S|_g$, where $c$
depends only on $n$, $n=4$ here. 
We know from the construction that $\Ricci(g) \to 0$ in $L^{2+\al}$ and the
other terms on the right hand side are bounded in $L^{2+\al}$ (because
$\partial g$ is bounded in
$L^{4+2\al}$ and $g,g^{-1}$ are bounded). Hence the right hand side is bounded
in $L^{2+\al}$ by a constant $c$ which doesn't depend on $i$, if $i \in \N$
is sufficiently large.
Also the terms $g^{ab}$ in front
of the first and second derivatives are
continuous, bounded and  satisfy $0< c_0|\psi|^2  \leq
g^{ij}\psi_i\psi_j \leq c_1|\psi|^2$ for $c_0,c_1$ independent of $i
\in \N$. 
Hence, using the $L^p$ theory (see for example \cite{GT} Theorem 9.11),
$|g(i)|_{W^{2,2+\al }(K)} \leq \int_{B_{1/100}(0)} |g(i)|^{2+\al } + c \leq \ti c$ on any smooth compact subset
$K \subseteq  B_{1/100}(0)$, where $\ti c$ is a constant which is
independent of $i$ (but does depend on $K$), and
in particular, using  the Theorem
of Rellich/Kondrachov (see for example \cite{Evans}, Theorem 1 of
Section 5.7), we see that $g(i)$  converges to $h$ 
in $W^{1,q  }(K)$ for any  $q  < \frac{4(2+\al )}{4 - (2+\al )}
=  \frac{4(2-\al ) + 8\al }{2-\al }= 4 + 8\frac{\al}{2-\al} ,$ for example $q = 4 +\si $  with $\si =8\frac{\al}{2-\al/2} > 2\alpha$  on smooth compact
subsets $K$ of $B_{1/100}.$  
We also have
\begin{eqnarray}
 h^{ab}\partial_a\partial_b g_{kl} = && (h^{ab}-g^{ab})\partial_a\partial_b g_{kl} +
g^{ab}\partial_a\partial_b g_{kl} \cr
= && (h^{ab}-g^{ab})\partial_a\partial_bg_{kl} 
 + (g^{-1} * g^{-1}* \partial g * \partial g)_{kl} -2 \Ricci(g)_{kl}  
\end{eqnarray}
Hence, for $g = g(i)$ (written in the coordinates given by
$\phi_{i,r}$) and $\ti g = g(j)$ (written in the coordinates given by
$\phi_{j,r}$) , $i,j \in \N$, we have
\begin{eqnarray}
 h^{ab} \partial_a\partial_b (g - \ti g)_{kl} 
 = && \curlL_{kl}:= (h^{ab}-g^{ab})\partial_a\partial_b g_{kl} -
 (h^{ab}-\ti g^{ab})\partial_a\partial_b \ti g_{kl} \cr
&& \ \ \
 + (g^{-1} * g^{-1}* \partial g * \partial g)_{kl} - ((\ti g)^{-1} *
 (\ti g)^{-1}* \partial \ti g * \partial \ti g)_{kl} \cr
&& -2 \Ricci(g)_{kl} +2\Ricci(\ti g)_{kl}. \label{fiddlyg1}
\end{eqnarray}
The right hand satisfies: for all $\ep >0$ there exists an $N\in
\N$ such that $|\curlL|_{L^{2+\al}(K)} \leq \ep$ if $i,j \geq N$, since 
$\partial g * \partial g$   converges to 
$\partial h * \partial h$  in $L^{2+\al}$ (since $g, g^{-1}$ converge to $h$,$h^{-1}$ in $W^{1,4+2\al}$ ),  
 $g(i) \to h, g(i)^{-1} \to h^{-1}$ in $C^{\al}$ on
smooth compact subsets   $K$ of $
B_{1/100}(0),$ and $\Rc(g_i) \to 0$ in $L^{2+\al}$ on compact subsets.

 Hence, we can rewrite \eqref{fiddlyg1} as 
\begin{eqnarray}
 h^{ab} \partial_a\partial_b (g(i) - g(j)) _{kl} 
 = && f(i,j)_{kl} \label{fiddlygnewer}
\end{eqnarray}
with $\int_K |f(i,j)|^{2+\al} \leq \ep(i,j)$, and $\ep(i,j) \leq \ep$ for
arbitrary $\ep >0$ if $i, j\geq N(\ep)$ is large enough.

Hence, using the $L^p$ theory again  (Theorem 9.11 in \cite{GT}), we get 
\begin{eqnarray}
 | g(i) -  g(j)|_{W^{2,2+\al}(\ti K)} && \leq c(\int
 |\curlL|_{K}^{2+\al} + \int  |g(i) -  g(j)
 |_{K}^{2+\al})\cr
&& \leq \de(i,j)
\end{eqnarray}
on any compact subset $ \ti K \subset \subset  K \subseteq B_{1/100}(0)$ 
where $\de: \N \times \N \to \R^+$ satisfies: for all $\ep>0$ there
exists an $N  =N(\ep) \in \N$ such that $\de(i,j) \leq \ep$ if $i,j \geq N$.
This implies $g(i)$ is a Cauchy sequence in  $W^{2,2+\al}(\ti K)$
and hence $g(i) \to h$ strongly in $W^{2,2+\al}(\ti K)$ and $h$ is
in $W^{2,2+\al}(\ti K)$.
Using these facts in \eqref{toodle}, and taking a limit as $i \to
\infty$ in $L^{2+\al}(K)$, we also see that
\begin{eqnarray}
h^{ab} \partial_a\partial_b (h_{kl})  = (h^{-1}*h^{-1}*\partial h
* \partial h)_{kl}  \label{inthearglp} 
\end{eqnarray}
must be satisfied in the $L^{2+\al}$ sense.
Standard  $L^p$ theory and  the Sobolev inequality now implies that $h \in W^{1,s}$ for all $ s\in(1,\infty)$.
We perform one iteration.
Let $h\in W^{1, 4 +\beta}$ (for example this holds in the first iteration step for $\beta = \al$).
Then this means that the right hand side of  \eqref{inthearglp}  is in  $   L^{2+ \beta/2}$. Hence, the $L^p$ theory tells us that on compact subsets, 
$ h \in W^{2,2+ \beta/2}$.
If $\beta/2\geq 2$ then the Sobolev embedding theorem tells us $h\in W^{1,q}$ for all $q \in (1,\infty)$ on compact subsets.
So assume $\beta/2 <2$. 
From the   Sobolev embedding theorem, 
$h \in W^{1, \frac{4(2+\beta/2)}{4-(2+\beta/2)}}=  W^{1,4 + \frac{4\beta}{2-\beta/2}}$
and hence  $h \in W^{1, 4 + 2 \beta}$.
Continuing in this way, we see   $h \in W^{1, q}$ for all $q \in (1,\infty)$ on compact subsets.
Hence the $L^p$ Theory tells us $h \in W^{2,q}$  for all $q \in (1,\infty)$ on compact subsets.
We may differentiate the equation using difference-quotients, and we get that 
 $h \in W^{3,q}$ for all $q \in (  1,\infty)$ on compact subsets, and so on.

We define $s_{i,r,t}= (\phi_{i,r})^{-1} \of \phi_{i,t} $ to be the the transition maps on $M_i.$ 
The equation satisfied by $s_i:= s_{i,r,t}$ 
is
\begin{eqnarray}
g(i)^{ab}\partial_a \partial_b s_i = 0 \label{eqnforsi}
\end{eqnarray}
For  $x \in U_{r}\cap U_{t}$ in $X,$ 
$s_{i,r,t}= (\phi_{i,r})^{-1} \of \phi_{i,t}:B_{2\ep}(v) \to \R^4$  is well defined 
 with $x = \phi_t(v)$ for some small $\ep>0$, and satisfies
$s_{i,r,t} \to s_{r,t} = (\phi_{r})^{-1} \of \phi_{t} $ weakly in $W^{2,2+\al}(B_{\ep}(v))$.  
Using the $L^p$ theory, we see that $s_i$ is bounded in
$W^{2,p}(B_{\ep/2}(v))$ for all $p \in (1,\infty)$, independently of $i$.
Hence, we have
\begin{eqnarray}
g(i)^{ab}\partial_a \partial_b (s_i -s_j) && =
-g(i)^{ab}\partial_a \partial_b s_j\cr
&& = (g(j) -g(i))^{ab} \partial_a \partial_b s_j
\end{eqnarray}
and the right hand side goes to $0$ in $L^p (B_{\ep/2}(v))$, as does
$s_i-s_j$, as $i,j \to \infty$, and hence
using the $L^p$ theory again, $s_i \to s$ in $W^{2,p}(B_{\ep/4}(v)$ for any $p\in (1,\infty).$
Hence, Equation \eqref{eqnforsi} holds in the limit, 
\begin{eqnarray}
h^{ab}\partial_a \partial_b s =0.
\end{eqnarray}
Using the regularity theory for elliptic equations (for example the
argument we used above to show that $h$ is smooth), we see that
$s$ is smooth. That is $(X,h,p)$ is a $C^{\infty}$ smooth Riemannian manifold. 
Using equation \eqref{Rmeqhar} we see that 
$(X,h,p)$ is  flat  and the volume growth estimates (b) of the assumptions, show us that $(X,h,p)$  has  Euclidean volume growth and hence $(X,h,p)$  is isometric to euclidean space with the standard metric.

The rest of the argument is identical to that of \cite{MSIM2}. The main points of the argument are: \\
It also holds, that\\
1) $(M_i,g(i),p_i) \to (X,h,p)$ in the pointed Gromov-Hausdorff $W^{1,4+\al}$  sense as
$i \to \infty$: for all $l$, there exists a smooth map $F_{i,l}:  P_l  \subseteq X \to
M_i$, such that $F_{i,l}(p) =p_i$ and $F_{i,l}: P_l\to M_i$
is a $C^{\infty}$ diffeomorphism onto its image, ${{}^h B_{l}}(p)
\subseteq P_l$,
and $F_{i,l}^*(g(i)) \to h$ in $W^{1,4+\al}(K)  $ strongly
for any compact set $K \subseteq U_l$, where $U_l$ is open in $X$.\\
 \\
2) The harmonic radius of $(\R^n,h)$ is infinity, since it is 
euclidean space.  This fact, the shown convergence properties and  the fact that $r_{h}(g(i))(y_i) =1$ now lead to  a contradiction.  

\end{proof}

\section{Appendix B}

The following   is essentially the {\it  Neck Lemma } of  Anderson/Cheeger \cite{AnCh2}  in the setting that Ricci curvature is bounded from below in the $L^p$ sense 
for $p>n$.   Their result and proof assumed the stronger condition that  Ricci curvature is bounded from below, and it is necessary to modify their proof at the appropriate points.  
\begin{thm}
Let  $ \si_0,\si_1>0,$ $ \infty > p > \frac{n}{2}$   be given constants. Then there exists an $R 
= R(\si_0,\si_1, p,n)>>1$  and an  $0< \ep_0  = \ep_0(\si_0,\si_1, p,n)$    such that  if   
 $(M^n,g)$   is a smooth, connected manifold with 
\begin{itemize}
\item[(i)]   $B_{g}(x_0,400R) \subset \subset M,$
\item[(ii)] $\int_{B_{g}(x_0,400R)} |\Rc_{-}|^p d\vol_g \leq \ep_0 $  
\item[(iii)]  $\si_1 r^n \geq  \vol_g(B_{g}(x,r)) \geq \si_0r^n $ for all $ x \in  B_{g}(x_0,400R)$   for all
$ r\in (0,400R)$,
\end{itemize}
then   
$   B_{g}(x_0,R)  \backslash B_{g}(x_0,1)$ is connected. 
\end{thm}
\begin{rem}
It follows that $   B_{g}(x_0,K)  \backslash B_{g}(x_0,k)$ is connected for all $K\geq R,0<k\leq 1$ with 
 $B_{g}(x_0,4K) \subset \subset M,$ since we can connect any point  in $B_{g}(x_0,K)  \backslash B_{g}(x_0,k)$
 to a point  in $B_{g}(x_0,R)  \backslash B_{g}(x_0,1)$ using a radial geodesic in $B_{g}(x_0,K)  \backslash B_{g}(x_0,k).$
\end{rem}

\begin{proof} 
Proof by contradiction. Assume $(M^n,g)$   is a smooth, connected manifold  
and everything is as in (i),(ii),(iii) of the theorem. 
We assume that $   B_{g}(x_0,R )  \backslash B_{g}(x_0,1 )$ has at least two components. 
We show for $R \geq  R_0( \si_0,\si_1, p,n)$ sufficiently large, that this assumption leads to a contradiction. 
 
We recall the following estimates from Petersen/Wei \cite{PeWei}   and  Dai/Wei \cite{DaiWei} (Book Chapter 5, proof of Theorem 5.3.1).  For $f(t):= \int_{\theta \in Z}\frac{\curlA(t,\theta)}{t^{n-1} } d\theta $  with $\curlA(t,\theta)d\theta$ being the volume form at $(t,\theta)$ written with respect to  Geodesic polar coordinates centred at some point $m \in B_{g}(x_0,20R),$ we have  
\begin{eqnarray}
\partt f^{\frac {1}{2p-1}}(t)  \leq  C(n,p) 
 \|\Rc_{-}\|_{L^p(\hat Z)}^{\frac{p}{2p-1}} t^{-\frac {n-1}{2p-1}}.
\end{eqnarray} 
for all $t\leq T_{\theta}$ where $ T_{\theta}$ is the first cutpoint of the unit-speed geodesic going in the direction $\theta$ and  $\hat Z:= \{   \exp(s \theta )  \ | \ s \in [0,200R],   \theta \in Z \}.$ 
We
define  $\curlA(t,\theta):=0$ for any $t \geq T_{\theta}$  and integrating the above inequality from $s<r$ to $r$ we get 
\begin{eqnarray} 
f^{\frac {1}{2p-1}}(r)-f^{\frac {1}{2p-1}}(s)
&& \leq  C_2(n,p)t^{\al(n,p)}|_s^r
 \|\Rc_{-}\|_{L^p(\hat Z)}^{\frac{p}{2p-1}}\cr
&& \leq  C_2(n,p)r^{\al} \|\Rc_{-}\|_{L^p(\hat Z)}^{\frac{p}{2p-1}},
\end{eqnarray} 
where $1>\al:= 1 -\frac{ n-1}{2p-1}>0$. Related, similar and other  inequalities  of this type   can   be found in Section 2.4 in  \cite {GTZLZ}.
Hence, for example  using $(|a|+ |b|)^{2p-1} \leq 2^{2p-1}(|a|^{2p-1} 
+  |b|^{2p-1}),$ we see that 
\begin{eqnarray} 
  \int_{\theta \in Z} \curlA(r,\theta) d\theta  
&&  \leq C_3(p,n)\frac{r^{n-1}}{s^{n-1}}  (
\int_{\theta \in Z} \curlA(s,\theta) d\theta  + s^{n-1}C_3(p,n) r^{\al(2p-1)}
\|\Rc_{-}\|_{L^p(\hat Z)}^{ p})
\end{eqnarray} 
holds, and hence
\begin{eqnarray} 
 C_4(p,n) \frac{s^{n-1}}{r^{n-1}} \int_{\theta \in Z} \curlA(r,\theta) d\theta  
&&  \leq   
\int_{\theta \in Z} \curlA(s,\theta) d\theta  + r^{\al(2p-1) + (n-1)}
     \|\Rc_{-}\|_{L^p(\hat Z)}^{ p } )
\end{eqnarray} 
where $C_4(p,n)>0$.
Keeping the convention 
$\curlA(t,\theta):=0$ for any $t \geq T_{\theta}$ 
we may integrate   with respect to $t$   again to obtain 

\begin{eqnarray}
&& \vol_g (  \{ \exp(v   t ) \ | \  v \in Z,  t \in [s,s+J] \}) \cr
&&   \ \ \ \geq  C_4(p,n) \frac{s^{n-1}}{(r+J)^{n-1}} \vol_g ( \{ \exp(v   t ) \ | \  v \in Z,  t \in [r,r+J]  \} ) \cr
&& \ \ \ \ \  \ \ \ -   J (r+J)^{\al(2p-1) + (n-1)} 
\|\Rc_{-}\|_{L^p(\hat Z)}^p 
\end{eqnarray}

In the following let $A_{g}(S,T) = \{ x \in M \ | \ S  \leq d_g(p,x) \leq T\}.$
We are assuming that there are two components of $B_{g}(p,10R) \backslash B_{g}(p,1)$. For any component $K$ and 
  any $ z\in K \cap (B_{g}(p,R))^c  $ we see, using the triangle inequality, 
 $B_{g}(z,R/2)  \subseteq  B_{g}(p,2R) \backslash B_{g}(p,R/4)$
and hence $ B_{g}(z,R/2) \subseteq K $  and hence,  using the volume estimates
$\vol_g( K \cap A_{g}(R/4,2R) ) \geq \vol_g( K  \cap  B_{g}(z,R/2)) = \vol_g(    B_{g}(z,R/2)) \geq \ti \si R^n.$ Hence, \\ ({\bf Fact 1}) : there must be a  $S \in [R/4, 2R]$ such that $\vol_g(  K \cap A_{g}(S,S+1) ) \geq \hat \si R^{n-1}  \geq c(n)\hat \si  S^{n-1}.$   \\

Let $K ,\hat K$ be different components of $ B_{g}(p,10R) \backslash B_{g}(p,1) $ and assume 
  $z \in K \cap \overline{B_{g}(p,L)}$ and $\hat z \in \hat K \cap B_{g}(p,R),$
   with $L>>1.$ 
 Any length minimising geodesic between $z,\hat z$ must lie  in $B_{g}(p,2R),$ by the triangle inequality.
If such a  geodesic doesn't go through $B_{g}(p,1),$ then it lies in $B_{g}(p,10R) \backslash B_{g}(p,1),$ and so 
 $K$ and $\hat K$ are connected,  leading to a contradiction.
 Hence for $z$ fixed,  every length minimising geodesic from $z$ to $\hat K \cap B_{g}(p,R)$ must go through 
 $B_{g}(p,1).$

We assume now, that $z \in  K \cap \boundary B_{g}(p,L) (\subseteq \overline{B_{g}(p,L)})$.
Let $\ga$ be any length minimising, speed one, geodesic starting at  $z,$  that makes it to $ (B_g(p,R/100))^c  \cap \hat K.$
Let $t$ be the last  time where $\ga(t) \in \overline{B_{g}(p,1)}.$ Since $\ga$ must go  through  $B_{g}(p,1),$ there must be a last time when $\ga(t) \in \overline{B_{g}(p,1)}.$ 
We must have $t \leq L+4$ and $t \geq L-4$ :  a) if $t> L+4,$ then we replace the curve before $\ga(t)$   by the radial geodesic from $z$ to $p$ followed by the radial geodesic  from $p$ to $\ga(t)$: this has length less than $L+2,$ which contradicts the fact that $\ga$ is length minimising up to $t$ and has length larger than $L+4.$ 
If $t \leq L-4,$ then $d_{g}(z,p) \leq d_g(z,\ga(t)) + d_g(p,\ga(t)) \leq L-3$, which is a contradiction, to $d_g(z,p) = L.$

A similar argument holds for the first time $\hat t,$ that $\ga(\hat t)$ reaches $\overline{B_{g}(p,1)}.$ 
Hence,  ({\bf Fact 2}):  if $v= \ga'(0)$  then 
\begin{eqnarray}
\exp(s v )  \in  A_g(s-L-10,s-L+10)  \label{importantinc}
\end{eqnarray}
 for all $s\in [2L, T_v]$ where $\exp(v T_v)$ is the first cut point of the geodesic starting at $z$ going in the direction $v$ 

 Let $\Theta$ be the set of length one directions $v$ such that there is some $V\leq T_v$ such that $\exp(s v)$ $s\in [0,V]$ is a length  minimising geodesic with $\ga(V) \in   A_{g}(S,S+1)\cap  \hat K.$ Note that for every 
 $b \in A_{g}(S,S+1)\cap  \hat K,$ there must be some $v \in \Theta$ and some $V \leq T_v$ such that 
 $\exp(Vv) = b.$ Also, \eqref{importantinc}, tells us that 
 $V \in (S+L-100,S+ L +100)$ for this $v \in \Theta,$ $V \leq T_v  $ with   $\exp(Vv) = b.$
 That is $ \hat K \cap A_{g}(S,S+1) \subseteq \{ \exp(v t ) \ | \  v \in \Theta,  t \in [S+L-100,S+L+100], t \leq T_v \} .$
 We continue with the convention from above, that $\curlA(t,\theta)$ is zero for all $t\geq T_{\theta}$ where $T_{\theta}$ is the first cut point
 of the geodesic going in the direction $\theta$.
   Fact 2 and Fact 1 show 
\begin{eqnarray}
\hat \si S^{n-1} && \leq \vol_g (\hat K \cap A_{g}(S,S+1)) \cr
 && \leq  \vol_g ( \{ \exp(v t ) \ | \  v \in \Theta,  t \in [S+L-100,S+L+100], t \leq T_v \}),
 \end{eqnarray}
 since according to  the above   we have
 $ \hat K \cap A_{g}(S,S+1) \subseteq \{ \exp(v t ) \ | \  v \in \Theta,  t \in [S+L-100,S+L+100], t \leq T_v \} .$

 Furthermore , we have $  \exp(v t ) \in B_{1000}(p),$ for all $t\in  [L-100,L+100],$ $v\in \Theta$ in view of the triangle inequality. 
Using  these facts,   the volume ratio lower and upper bounds,  the volume  estimates of Wei/Petersen ,  Dai/Wei from the start of this proof,     we see
\begin{eqnarray}
 && \si_1 {1000}^n \cr
 &&  \geq \vol_g ( B_{1000}(p))  \cr
 && \geq  \vol_g ( \{ \exp(v t ) \ | \  v \in \Theta,  t \in [ L-100,L+100]     \})   \cr
 && \geq C_4(p,n) \frac{(L-100)^{n-1}}{ (S+L+100)^{n-1} } \vol_g ( \{ \exp(v t ) \ | \  v \in \Theta,  t \in  [S+L-100,S +L+100]    \}) \cr
 && \ \ \  - 200(S+L+100)^{\al(2p-1)+(n-1)} \|\Rc_{-}\|_{L^p(B_{g}(p,200R))}^p \cr
 && \geq C_4(p,n)\hat{\sigma} S^{n-1} \frac{(L-100)^{n-1}}{ (S+L+100)^{n-1} }   -   200(S+L+100)^{\al(2p-1)+(n-1)}\ep_0
 \end{eqnarray}
  
 If we choose $R= 50L \geq 10^{10}$ (remembering $S \in [R/4,2R]$ ), then we have
   \begin{eqnarray}
   \si_1 {1000}^n  \geq C_5(p,n, \si_0,\si_1) R^{n-1} -R^{k(p,n)}\ep_0 
 \end{eqnarray}
 
  which leads to a contradiction for $R^{n-1} \geq   \frac{\si_1 2000^n}{C_5(p,n,\si_0,\si_1)} $ and
  $\ep_0 \leq \frac{\si_1}{R^{k(p,n)}}$.

\end{proof}

\newpage
 
\noindent {\it Conflict of interest statement}: There is no conflict of interest. \\
{\it Data availability statement }: No datasets were generated or analysed during the current study.\\


\begin{thebibliography}{9}
 

\bibitem{Anderson}
\text{ M. Anderson  }
{\it  Convergence and rigidity of manifolds under Ricci curvature bounds},
   Inventiones mathematicae,102,
  pp. 429-445  (1990)
 
 \bibitem{AnCh}
\text{  M. Anderson and J.  Cheeger },
 {\it $C^\alpha$-compactness for manifolds with Ricci curvature and injectivity radius bounded below}, 
 Journal of Differential Geometry 
vol. 35, n. 2,
 pp. 265 -- 281, (1992)
 

\bibitem{AnCh2}
\text{M. Anderson and  J.  Cheeger  } {\it Diffeomorphism finiteness for manifolds with ricci curvature and $L^{n/2}$-norm of curvature bounded,} Geometric and Functional Analysis 1, 231–252 (1991). 

\bibitem{APTE}
 \text{M.~Apte},
{\it Sur certaines classes caract\`eristiques des vari\`et\`es K\'ahleriennes compactes}, Comptes rendus hebdomadaires des seances de l academie des sciences, \textbf{240} , 149--151, (1955).
\bibitem{Bam}
 \text{R.~Bamler},
 {\it Entropy and heat kernel bounds on a Ricci flow background,}
  	arXiv:2008.07093,  
 (2020)
\bibitem{RBQZ}
 \text{R.~Bamler} and  \text{Q. ~Zhang}, 
{\it Heat kernel and curvature bounds in Ricci flows with bounded scalar curvature}, Advances in Mathematics, \textbf{319} ,396--450., (2017).

\bibitem{RBQZ2}
 \text{R.~Bamler} and  \text{Q. ~Zhang}, 
{\it Heat kernel and curvature bounds in Ricci flows with bounded scalar curvature--Part II}, Calc. Var. 58, 49 (2019).  

 \bibitem{Besse}
 \text{A.   Besse }, 
Einstein manifolds  ,Springer, ISBN 978-3-540-74120-6 (2007) 


\bibitem{Carron}
 \text{G.~Carron},
{\it In\'egalit\'es isop\'erim\'etriques de Faber-Krahn et cons\'equences}, Actes de la table ronde de g\'eom\'etrie diff\'erentielle (Luminy, 1992), Collection SMF S\'eminaires et Congres, \textbf{1} , 205--232, (1996).

\bibitem{DaiWei} X. Dai and  G. Wei, {\it Comparison Geometry for Ricci curvature}, online book, \\

\bibitem{Evans} L. Evans,   \emph{Partial differential equations},
  Second edition, Graduate Studies Vol. 19, AMS (2010)
  
\bibitem{GT} D. Gilbarg and  N. Trudinger,   \emph{Elliptic Partial
    Differential Equations of Second Order}, 
Third revised edition, Springer, (2000)


\bibitem{SHA}
\text{S.~Hamanaka}, 
{\it Ricci flow with bounded curvature integrals},Pacific Journal of Mathematics, Vol. 314, No. 2, (2021)





\bibitem{Ham}
\text{R. Hamilton,}
 {\it Three-manifolds with positive Ricci curvature},  
  Journal of Differential Geometry, 255-306, (1982)  
 



\bibitem{LiuSim1} 
\text{J. Liu} and \text{M. Simon}
{\it Local integral formula for solutions to Ricci flow with bounded $L^p$ scalar curvature}, 
arxiv preprint (2024) 

\bibitem{GP}
\text{G.~Perelman}, 
{\it The entropy formula for the Ricci flow and its geometric applications},   arXiv: math/0211159, (2002)

\bibitem{PeWei}
\text{P. Petersen and  G. Wei,  } 
{\it Relative Volume Comparison with Integral Curvature Bounds}, GAFA, Geom. funct. anal. 7, 1031–1045 (1997)




\bibitem{MSIM1}
\text{M.~Simon},
{\it Some integral curvature estimates for the Ricci flow in four dimensions}, Communications in analysis and geometry , Internat. Press, Bd. 28. 3, S. 707-727, (2020)

\bibitem{MSIM2}
\text{M.~Simon},
{\it Extending four dimensional Ricci flows with bounded scalar curvature},  Communications in analysis and geometry , Internat. Press, Bd. 28 ,7, S. 1683-1754, (2020)

\bibitem{ST16} M. Simon  and  P. Topping, 
{\it Local control on the geometry in 3D Ricci flow}, Journal of Differential Geometry, 122(3) 467-518 (2022).  

\bibitem{GTZLZ}
\text{G.~Tian} and \text{Z.~Zhang},
{\it Regularity of K\"ahler-Ricci flows on Fano manifolds, Acta Mathematica}, \textbf{216} , 127--176, (2016).




\bibitem{RGY} 
\text{R. ~Ye}, 
{\it The logarithmic Sobolev and Sobolev inequalities along the Ricci flow}, Communications in Mathematics and Statistics, \textbf{2} , 363--368, (2014).


\bibitem{QSZ} 
\text{Q.~Zhang},
{\it A uniform Sobolev inequality under Ricci flow}, International Mathematics Research Notices, (2007).


\bibitem{QSZ2} 
\text{Q.~Zhang},
{\it Erratum to : A uniform Sobolev inequality under Ricci flow}, International Mathematics Research Notices, (2007).
 





\bibitem{QSZ1} 
\text{Q.~Zhang},
{\it Bounds on volume growth of geodesic balls under Ricci flow}, Mathematical Research Letters, \textbf{19} , 245--253, (2012).


\end{thebibliography}
\end{document}